\documentclass[a4paper,reqno,12pt]{amsart}
\usepackage{amsmath}
\usepackage{amssymb}
\usepackage{amsthm}
\usepackage[noadjust]{cite}
\usepackage{mathrsfs}
\usepackage[all,cmtip]{xy}
\usepackage{cite}
\usepackage{cases}
\usepackage{enumitem}
\usepackage{indentfirst}
\usepackage{hyperref}
\usepackage{comment}

\usepackage{geometry}
\usepackage{xy}
\allowdisplaybreaks[4]

\newtheorem{thm}{Theorem}[section]

\newtheorem{lem}[thm]{Lemma}

\newtheorem{conj}[thm]{Conjecture}

\theoremstyle{definition}
\newtheorem{definition}[thm]{Definition}

\theoremstyle{remark}
\newtheorem{remark}[thm]{Remark}

\numberwithin{equation}{section}

\newcommand{\bN}{\mathbb{N}}

\newcommand{\bQ}{\mathbb{Q}}
\newcommand{\bR}{\mathbb{R}}
\newcommand{\bZ}{\mathbb{Z}}

\newcommand\cC{\mathcal{C}}

\newcommand\cE{\mathcal{E}}
\newcommand\cF{{\mathcal{F}}}

\newcommand\cO{{\mathcal{O}}}
\newcommand\cP{\mathcal{P}}

\newcommand\cS{{\mathcal{S}}}
\newcommand\cSS{{\mathcal{S}}{\mathcal{S}}}

\newcommand{\rounddown}[1]{\lfloor{#1}\rfloor}

\newcommand{\ff}{{f^{\prime}}}
\newcommand{\BB}{{B^{\prime}}}

\newcommand{\DD}{{D^{\prime}}}

\newcommand{\MM}{{M^{\prime}}}

\newcommand{\UU}{{U^{\prime}}}

\newcommand{\XX}{{X^{\prime}}}
\newcommand{\YY}{{Y^{\prime}}}
\newcommand{\ZZ}{{Z^{\prime}}}

\newcommand{\ZZZ}{Z^{\prime \prime}}

\newcommand{\ZZZZ}{Z^{\prime \prime \prime}}

\newcommand{\BBi}{{B^{\prime}_i}}
\newcommand{\MMi}{{M^{\prime}_i}}

\newcommand{\oA}{{\overline{A}}}
\newcommand{\oB}{{\overline{B}}}

\newcommand{\oM}{{\overline{M}}}
\newcommand{\oN}{{\overline{N}}}

\newcommand{\oV}{{\overline{V}}}
\newcommand{\oX}{{\overline{X}}}
\newcommand{\oSigma}{{\overline{\Sigma}}}
\newcommand{\oGamma}{{\overline{\Gamma}}}
\newcommand{\oDelta}{{\overline{\Delta}}}

\newcommand{\mult}{\operatorname{mult}}
\newcommand{\Nklt}{\operatorname{Nklt}}

\newcommand{\Supp}{\operatorname{Supp}}

\newcommand{\vol}{\operatorname{vol}}
\newcommand{\Ivol}{\operatorname{Ivol}}
\newcommand{\Div}{\operatorname{Div}}
\newcommand{\Nlc}{\operatorname{Nlc}}
\newcommand{\red}{\operatorname{red}}

\hypersetup{
	colorlinks=true,
	citecolor=red,
	linkcolor=blue,
	urlcolor=green,
}

\begin{document}
\title{BOUNDEDNESS OF STABLE MINIMAL MODELS WITH KLT SINGULARITIES}
\date{\today}

\author{MINZHE ZHU}
\address{MINZHE ZHU, School of Mathematical Sciences, Fudan University, Shanghai, 200433, China}
\email{zhumz20@fudan.edu.cn}

\begin{abstract}
We investigate the singularities and boundedness of a special kind of algebraic varieties so-called stable minimal models, which are constructed and studied  by Birkar in \cite{BVGP,moduliofalgebraicvarieties}. Given a klt stable minimal model with bounded relative volume, if we fix the dimension, Iitaka volume, and a DCC set controlling coefficients, then we show that the singularities of the klt stable minimal model can be controlled uniformly. Furthermore, we prove that with certain bounded data, stable minimal models with klt singularities form a bounded family. 
\end{abstract}

\maketitle
\pagestyle{myheadings} \markboth{\hfill M.Z.~Zhu
	\hfill}{\hfill Boundedness of stable minimal models with klt singularities\hfill}

\tableofcontents

\section{Introduction}
We work over an algebraically closed field $k$ of characteristic zero.
~\\

Boundedness properties of algebraic varieties and singularities have been extensively studied in recent years. Regarding algebraic varieties, such as canonically polarized varieties, polarized Calabi-Yau varieties, and Fano varieties, proving boundedness properties is the first step of constructing their potential moduli spaces, see \cite{ACCLCT,HX15,HMX18,BABI,BABII,BVGP,moduliofalgebraicvarieties,GOPV}, etc. Moreover, they are also essential tools to prove effective birationalities and effective Iitaka fibrations of algebraic varieties in higher dimensional birational geometry, see \cite{ACCLCT, effectiveIitaka, BABI,GOPV,CHL23}, etc. As for singularities, Xu and Zhuang proved boundedness of K-semistable log Fano cone singularities recently, and it is a crucial step to prove discreteness of local volumes of klt singularities, see \cite{HLQ23,boundednessofsingularitiesII,boundednessofsingularitiesIII,boundednessofsingularitiesI}, etc.

Since the boundedness results of canonically polarized varieties and polarized Calabi-Yau varieties have been established, it remains to investigate algebraic varieties with intermediate Kodaira dimension. Recent work related to this aspect can be found in \cite{filipazziSvaldi20,FHS2021boundedness,BVGP,moduliofalgebraicvarieties,jiao2022boundedness,HH23boundedness,jiang23boundedness,filipazzi2024boundedness,Iitakavolume}, etc. In the absence of a natural polarization, unlike canonically polarized varieties, we need to add a polarization. A suitable option is given by Birkar. He defines stable minimal models to be such varieties plus polarization, which generalizes both KSBA-stable varieties and polarized Calabi-Yau varieties. The main purpose of this article is to study boundedness properties of stable minimal models with klt singularities. Here we recall the definition of klt stable minimal models. Note that in general stable minimal models can be defined in slc case. 

\begin{definition}\label{stable minimal models}{(Klt stable minimal models, \cite{moduliofalgebraicvarieties})}
	A \textit{klt stable minimal model} $(X,B), A$ consists of $(X,B)$ and an $\bR$-divisor $A\geq0$ where
	\begin{itemize}
		\item $(X,B)$ is a projective klt pair,
		\item $K_X+B$ is semi-ample defining a contraction $f: X\to Z$, and
		\item $K_X+B+tA$ is ample for some $t>0$, i.e. $A$ is ample over $Z$.
	\end{itemize}
	
	Moreover, if in addition $K_X+B+A$ is ample, then we call $(X,B), A$ a \textit{klt strongly  stable minimal model}.
\end{definition}

For the motivation and examples of stable minimal models, we recommend readers to \cite{moduliofalgebraicvarieties,familiesofgeneraltype} and the references therein.

Next we give the definition of families of stable minimal models with certain data.

\begin{definition}{(Families of klt stable minimal models, \cite{moduliofalgebraicvarieties})}
	Let $d\in \bN$, $u, v\in \bR^{>0}$, and $\Phi\subset \bR^{\geq0}$.
	\begin{enumerate}[itemsep=13pt]
		\item A \textit{$(d, \Phi)$-klt stable minimal model} is a klt stable minimal model $(X,B), A$ such that
		\begin{itemize}
			\item $\dim X=d$, and
			\item $B, A\in \Phi$, which means that the coefficients of $B$ and $A$ are in $\Phi$.
		\end{itemize}
		
		Let $\cS_{klt}(d,\Phi)$ denote the set of all $(d,\Phi)$-klt stable minimal models. Denote by $\cSS_{klt}(d,\Phi)\subset \cS_{klt}(d,\Phi)$ the subset of all klt strongly stable minimal models.
		\item A \textit{$(d, \Phi, \leq u, v)$-klt stable minimal model} is a $(d, \Phi)$-klt stable minimal model $(X,B),A$ such that
		\begin{itemize}
			\item $\vol(A|_F)\leq u$, where $F$ is a general fiber of $f: X\to Z$, and
			\item $\Ivol(K_X+B)=v$ (see Definition \ref{Iitaka volume} for the definition of Iitaka volumes of $\bR$-divisors).
		\end{itemize}
		
		Let $\cS_{klt}(d,\Phi,\leq u, v)$ consist of all $(d, \Phi, \leq u,v)$-klt stable minimal models. Similarly define \textit{$(d, \Phi, \leq u, \leq v)$-klt stable minimal model} and $\cS_{klt}(d,\Phi,\leq u, \leq v)$ by replacing the condition ``$\Ivol(K_X+B)=v$'' with ``$\Ivol(K_X+B)\leq v$''.

		\item When $0$ is not an accumulation point of $\Phi$ (e.g. when $\Phi$ is DCC), we say that a subset $\cE\subseteq \cS_{klt}(d,\Phi)$ forms a bounded family if there is a positive integer $r$ such that for each $(X,B),A\in \cE$, there is a very ample divisor $H$ on $X$ such that $H^d\leq r$ and $(K_X+B+A)\cdot H^{d-1}\leq r$.
	\end{enumerate}
\end{definition}

Given a klt stable minimal model $(X,B),A\to Z$ in the family $\cS_{klt}(d,\Phi,\leq u,  v)$ with fixed $d,\Phi,u,v$, we can prove that the base $Z$ is in a bounded family by Theorem \ref{effective adjunction formula} and \cite[Theorem 1.4]{BVGP}, and the general fiber $F$ is also in a bounded family by \cite[Theorem 6.2]{GOPV}. 
It is natural to ask whether or not the total space $(X,B)$ belongs to a bounded family.
\begin{conj}\label{conj1}
	Let $d\in \bN$, $u,v\in \bR^{>0}$, and $\Phi\subset \bR^{\geq0}$ be a DCC set. If $(X,B),A\in \cS_{klt}(d,\Phi,\leq u, v)$, then the set of such $(X,B)$ forms a bounded family.
\end{conj}

Jiao proved that if $u,v\in \bQ^{>0}$, and $\Phi\subset \bQ^{\geq 0}$, then $\cS_{klt}(d,\Phi, u,  v)$ is a birationally bounded family (cf. \cite[Theorem 1.2]{jiao2022boundedness}). Note that $\cS_{klt}(d,\Phi, u,  v)$ is defined similarly to $\cS_{klt}(d,\Phi,\leq u, v)$ by replacing the condition ``$\vol(A|_F)\leq u$'' with ``$\vol(A|_F)= u$''. Moreover, in \cite[Theorem 1.4]{HH23boundedness}, Hashizume and Hattori proved the $\bQ$-coefficient version of Conjecture \ref{conj1}, in the case that $\vol(A|_F)$ is fixed, and the Iitaka dimension of stable minimal models is one, that is, $(X,B),A\to Z\in \cS_{klt}(d,\Phi, u,  v)$ with $\dim Z=1$.

In \cite[Theorem 1.9]{BVGP}, Birkar proved $\bQ$-coefficient version of Conjecture  \ref{conj1} for klt strongly stable minimal models $(X,B),A$ with additional conditions that $\vol(A|_F)$ is fixed, and $\vol(K_X+B+A)$ is bounded from above (note that the notation ``$\cS_{klt}(d,\Phi,u,v,<w)$'' in \cite[Theorem 1.9]{BVGP} represents the family of klt strongly stable minimal models defined in \cite{moduliofalgebraicvarieties} and this article). Therefore, we can expect that \cite[Theorem 1.9]{BVGP} also holds for klt stable minimal models. The following conjecture is a weak version of Conjecture \ref{conj1}.

\begin{conj}\label{conj2}
	Let $d\in \bN$, $u,v,w\in \bR^{>0}$, and $\Phi\subset \bR^{\geq0}$ be a DCC set. Let $(X,B),A\in \cS_{klt}(d,\Phi,\leq u, v)$ satisfy that $\vol(K_X+B+A)<w$. Then the set of such $(X,B),A$ forms a bounded family.
\end{conj}

In \cite{moduliofalgebraicvarieties}, Birkar proved boundedness of $\bQ$-coefficient slc stable minimal models $(X,B),A$ of dimension $d$ assuming that the intersection number $(K_X+B)^i\cdot A^{d-i}$ is fixed for any $0\leq i\leq d$, see \cite[Theorem 1.12, Lemma 4.12]{moduliofalgebraicvarieties}. In light of his result, we prove a special case of Conjecture \ref{conj1}.

\begin{thm}\label{mainthm1}
	Let $d\in \bN$, $u,v,w\in \bR^{>0}$, and $\Phi\subset \bR^{\geq 0}$ be a DCC set. Let $(X,B),A\in \cS_{klt}(d, \Phi,\leq u,v)$ satisfy that the intersection number $(K_X+B)^i\cdot A^{d-i}\leq w$ for any $0\leq i\leq d$. Then the set of such $(X,B),A$ forms a bounded family.
\end{thm}

Note that the condition on $\vol(A|_F)$ can  be removed in Theorem \ref{mainthm1} (see Remark \ref{remove condition}).

The main difficulty of the proof comes from $\bR$-coefficients, which is not dealt in \cite{BVGP,moduliofalgebraicvarieties}. We will explain this more precisely in the sketch of proofs. Indeed, Birkar expected that \cite[Theoerm 1.12, Theorem 1.14]{moduliofalgebraicvarieties} could be generalized to the real coefficient case (cf. \cite[Subsection 11.11]{moduliofalgebraicvarieties}).  Theorem \ref{mainthm1} confirms that $\bR$-coefficient version of \cite[Theoerm 1.12]{moduliofalgebraicvarieties} holds for klt stable minimal models. However, our method can not be applied to the slc case. 

One of the key steps of the proof of Theorem \ref{mainthm1} is controlling the singularities of klt stable minimal models in $\cS_{klt}(d, \Phi,\leq u,v)$ uniformly.

\begin{thm}\label{singularities of klt stable minimal models}
	Let $d\in \bN$, $u, v\in \bR^{>0}$, and $\Phi\subset \bR^{\geq0}$ be a DCC set. Then there exists $\epsilon\in \bR^{>0}$ satisfying the following. 
	
	If $(X,B), A\in \cS_{klt}(d, \Phi, \leq u, v)$, then $(X,B)$ is $\epsilon$-lc.
\end{thm}

In light of Theorem \ref{singularities of klt stable minimal models}, if we restrict ourselves to $\epsilon$-lc stable minimal models for some fixed $\epsilon\in \bR^{>0}$, then in Theorem \ref{mainthm1} the condition ``$\Ivol(K_X+B)=v$" can be replaced with ``$\Ivol(K_X+B)\leq v$". However, in this case, $\Phi$ should be a finite set by technical reasons.

\begin{thm}\label{mainthm2}
	Let $d\in \bN$, $\epsilon, u,v,w\in \bR^{>0}$, and $\Phi\subset \bR^{\geq0}$ be a finite set. Let $(X,B),A\in\cS_{klt}(d, \Phi,\leq u,\leq v)$ satisfy that $(X,B)$ is $\epsilon$-lc, and the intersection number $(K_X+B)^i\cdot A^{d-i}\leq w$ for any $0\leq i\leq d$. Then the set of such $(X,B),A$ forms a bounded family.
\end{thm}
~\\
\paragraph{\textbf{Sketch of proofs}} 
We start with Theorem \ref{singularities of klt stable minimal models}. Let $(X,B),A\in \cS_{klt}(d,\Phi,\leq u,v)$. Let $f:X\to Z$ be a contraction defined by the semi-ample $\bR$-divisor $K_X+B$. First we generalize the effective adjunction formula in \cite[Lemma 7.4]{BVGP} to the $\bR$-coefficient case (see Theorem \ref{effective adjunction formula}). The main new tool of proving Theorem \ref{effective adjunction formula} is the uniform rational polytope for adjunction formula (cf. \cite[Theorem 3.3]{effectiveadjunctionformulawithrealcoefficients}).

 Next, by Theorem \ref{effective adjunction formula}, there is an adjunction formula $K_X+B\sim_{\bR}f^*(K_Z+B_Z+M_Z)$ such that $(Z,B_Z+M_Z)\in \cF_{gklt}(\dim Z,\Psi,v)$ for some fixed DCC set $\Psi\subset \bR^{\geq 0}$ (see Definition \ref{familyofgpairs} for the definition of $\cF_{gklt}(\dim Z,\Psi,v)$). Then by \cite[Theorem 1.5]{BVGP}, there exists a fixed $\delta\in \bR^{>0}$ such that $(Z,B_Z+M_Z)$ is generalized $\delta$-lc. 

Take a prime divisor $D$ over $X$. If $D$ is horizontal over $Z$, then $D$ determines a prime divisor $S$ over the general fiber $F$ of $f:X\to Z$. Since $(F,B_F)$ is a klt log Calabi-Yau pair, and $B_F\in \Phi$, then $(F,B_F)$ is $\tau$-lc for some fixed $\tau\in \bR^{>0}$ by \cite[Lemma 2.48]{BABI}. Hence in this case, $a(D,X,B)=a(S,F,B_F)\geq \tau$. On the other hand, if $D$ is vertical over $Z$, take resolutions $\pi:\XX\to X$ and $\mu:\ZZ\to Z$ such that $\ff:\XX\dasharrow \ZZ$ is a morphism, $D$ is a prime divisor on $\XX$ and $E=\ff(D)$ is a prime divisor on $\ZZ$. Let $K_{\XX}+\BB=\pi^*(K_X+B)$. By the adjunction formula, $(\XX,\BB+\delta \ff^*E)$ is lc over the generic point of $E$, hence $a(D,X,B)=a(D,\XX,\BB)\geq 1-\delta$. Choose $\epsilon=\min\{\tau,\delta\}$, and we finish the proof.
~\\

Next we move on to Theorem \ref{mainthm1}. Let $(X,B),A\in \cS_{klt}(d, \Phi,\leq u,v)$ satisfy that the intersection numbers $(K_X+B)^i\cdot A^{d-i}\leq w$. By Theorem \ref{singularities of klt stable minimal models}, we can assume that $(X,B)$ is $\epsilon$-lc for some fixed $\epsilon\in\bR^{>0}$. As in the proof of \cite[Theorem 4.1]{moduliofalgebraicvarieties}, 
 the key point of the proof is to find a fixed $\lambda\in\bR^{>0}$ such that $(X,B+\lambda A)$ is lc and $K_X+B+\lambda A$ is nef, because then $(X,B+\frac{\lambda}{2}A)$ is $\frac{\epsilon}{2}$-lc, $K_X+B+\frac{\lambda}{2} A$ is ample, and $\vol(K_X+B+\frac{\lambda}{2}A)$ is bounded from above by the upper bound of the intersection numbers, hence $(X,B),A$ belongs to a bounded family by \cite[Theorem 1.3, Theorem 1.6]{ACCLCT}. 
 
 In the proof of \cite[Theorem 4.1]{moduliofalgebraicvarieties}, Birkar showed that there is a fixed $p\in \bN$ and an adjunction formula  $K_X+B\sim_{\bR}f^*(K_Z+B_Z+M_Z)$, such that $L:=p(K_Z+B_Z+M_Z)$ is very ample. Then take a general member $T$ of $|L|$, and let $S$ be the pullback of $T$ on $X$. There is a new stable minimal model $(S,B_S),A_S:=A|_S$ and a fibration $S\to T$. The intersection numbers $(K_S+B_S)^i\cdot A_S^{d-1-i}$ are controlled by $(K_X+B)^i\cdot A^{d-i}$. Therefore, by induction $(S,B_S),A_S$ is bounded. Then using \cite[Lemma 4.7]{moduliofalgebraicvarieties}, the inversion of adjunction formula, and the length of extremal rays, we can find a fixed $\lambda\in \bR^{>0}$ such that $(X,B+\lambda A)$ is lc, and $K_X+B+\lambda A$ is nef.  
 
 However, in $\bR$-coefficient case, $K_Z+B_Z+M_Z$ is not a $\bQ$-divisor, hence there is no such $p\in \bN$ that makes $p(K_Z+B_Z+M_Z)$ very ample. An alternative way is to decompose $K_Z+B_Z+M_Z$ uniformly, i.e., $K_Z+B_Z+M_Z=\sum_{i=1}^lr_i(K_Z+B_{i,Z}+M_{i,Z})$ such that $r_i\geq 0$, $\sum_{i=1}^l r_i=1$, and there is a fixed $p\in \bN$ such that $L_i:=p(K_Z+B_{i,Z}+M_{i,Z})$ is very ample for any $i$ (see Theorem \ref{decomposition of ample divisor}). But if we take a general member $T$ of $|L_1|$, and denote by $(S,B_S),A_S\to T$ the induced stable minimal model, then there is no way to control the upper bound of the intersection numbers $(K_S+B_S)^j\cdot A_S^{d-1-j}$ for $0\leq j\leq d-1$, which means that $\vol(K_S+B_S+A_S)$ cannot be controlled. Hence we cannot use induction in this way.
 
To solve this problem, our new idea is to consider the image of non-lc locus of $(X,B+tA)$ in $Z$, which we denote by $P(t)$, i.e., $P(t)=f(\Nlc(X,B+tA))$, and then proving the following two statements inductively:
\begin{enumerate}[wide=13pt]
	\item There exists a fixed $\lambda_k\in \bR^{>0}$ such that $\dim P(\lambda_k)\leq \dim Z-k$.
	\item There exists a fixed $\mu_k\in \bR^{>0}$ such that $K_X+B+\mu_k A$ is nef in dimension $k-1$ over $Z$ (see Definition  \ref{relatively nefness} for the definition of relatively nefness). 
\end{enumerate}

Let $F$ be the general fiber of $X\to Z$, and $(F,B_F), A_F$ be the restriction of $(X,B),A$ on $F$. Then $(F,B_F),A_F$ is a polarized log Calabi-Yau pair. By \cite[Theoem 6.2]{BVGP}, $(F,B_F), A_F$ is bounded. By \cite[Theorem 1.8]{BABII}, there is a fixed $\lambda_1\in \bR^{>0}$ such that $(F,B_F+\lambda_1A_F)$ is lc. By the inversion of adjunction \cite{inversionofadjunction,haconinversion}, $\dim P(\lambda_1)\leq \dim Z-1$. Since $K_X+B$ is semi-ample and $A$ is ample over $Z$, automatically $K_X+B+A$ is nef in dimension $0$ over $Z$. Therefore, the two statements hold for $k=1$.

When $k\geq 1$, assume there is a fixed $\lambda_k\in\bR^{>0}$ such that $\dim P(\lambda_k)\leq \dim Z-k$.  Let $T$ be a general complete intersection of dimension $k$, and $S$ be the pullback of $T$ on $X$. The adjunction formula gives a pair $(S,B_S+\lambda_kA_S)$ on $S$. By \cite[Theorem 1.2, Corollary 1.4]{fujinoadjunction}, we have $\Nlc(S,B_S+\lambda_k A_S)=S\cap \Nlc(X,B+\lambda_k A)$. Since $\dim P(\lambda_k)\leq \dim Z-k$, we deduce that $\Nlc(S,B_S+\lambda_k A_S)$ is contained in finitely many fibers of $S\to T$. By the standard method of using the length of extremal rays, there is a fixed $\mu_{k+1}\in \bR^{>0}$, such that $K_S+B_S+\mu_{k+1} A_S$ is globally nef. Thus $K_X+B+\mu_{k+1}A$ is nef over $Z$ in dimension $k$.

On the other hand, assume there is a fixed $\mu_k\in \bR^{>0}$ such that $K_X+B+\mu_k A$ is nef in dimension $k-1$ over $Z$. First by Theorem \ref{effective adjunction formula}, we construct an adjunction formula $K_X+B\sim_{\bR}f^*(K_Z+B_Z+M_Z)$, and decompose $K_Z+B_Z+M_Z$ uniformly, i.e., $K_Z+B_Z+M_Z=\sum_{i=1}^lr_i(K_Z+B_{i,Z}+M_{i,Z})$ such that $r_i\geq 0$, $\sum_{i=1}^lr_i=1$, and there is a fixed $p\in \bN$ such that $L_i:=p(K_Z+B_{i,Z}+M_{i,Z})$ is very ample for any $i$. Then there exists a fixed $m\in \bN$ such that $m(K_Z+B_Z+M_Z)-L_1$ is nef. Take $T$ as a general complete intersection of dimension $k-1$ cut by hypersurfaces in $|L_1|$. Let $S$ be the pullback of $T$ on $X$. The adjunction formula gives a pair $(S,B_S+\mu_k A_S)$. By the definition of relative nefness in dimension $k-1$, $K_S+B_S+\mu_k A_S$ is nef. Decreasing $\mu_k$, we can assume that $K_S+B_S+\mu_k A_S$ is ample. Therefore, $(S,B_S),\mu_kA_S$ is a strongly stable minimal model. Although the intersection numbers $(K_S+B_S)^i\cdot A_S^{\dim S-i}$ are still not controlled, we can control $\vol(K_S+B_S+\mu_k A_S)$, because $m(K_Z+B_Z+M_Z)-L_1$ is nef, and $K_X+B+\mu_k A$ is nef in dimension $k-1$ over $Z$. Therefore, by some calculations, we prove that $\Ivol(K_S+B_S)$ is in a fixed finite set and $\vol(K_S+B_S+\mu_k A_S)$ is bounded from above. By the lower bound of lc thresholds in strongly stable minimal models (see Theorem \ref{lct for strong minimal model}), there is a fixed $\lambda_k\in \bR^{>0}$ such that $(S,B_S+\lambda_k A_S)$ is lc. By the inversion of adjunction  \cite{inversionofadjunction,haconinversion}, we conclude that $P(\lambda_k)\cap T=\emptyset$. Note that $T$ is a general complete intersection of dimension $k-1$, hence $\dim P(\lambda_k)\leq s-k$.

Applying induction  on $k$, we can find such $\lambda$ that $(X,B+\lambda A)$ is lc outside from finitely many fibers of $f:X\to Z$ and $K_X+B+\lambda A$ is nef. Applying the standard method of using the length of extremal rays again, decreasing $\lambda$, we conclude that $(X,B+\lambda A)$ is globally lc. Therefore, we finish the proof of Theorem \ref{mainthm1}.
~\\

At last, we give a brief sketch of the proof of Theorem \ref{mainthm2}. Let $(X,B),A\in \cS(d,\Phi,\leq u,\leq v)$ satisfy that the intersection numbers $(K_X+B)^i\cdot A^{d-i}\leq w$. Let $f: X\to Z$ be the contraction defined by the semi-ample $\bR$-divisor $K_X+B$. By Theorem \ref{finite effective adjunction}, we construct an adjunction formula $K_X+B\sim_{\bR}f^*(K_Z+B_Z+M_Z)$ such that $(Z,B_Z+M_Z)\in \cF_{gklt}(\dim Z, J, \leq v)$ for some fixed finite set $J\subset \bR^{\geq0}$, and $(Z,B_Z+M_Z)$ is generalized $\delta$-lc for some fixed $\delta\in \bR^{>0}$. Then applying the proof of \cite[Theorem 1.3]{filipazzi2020some}, we show that $\vol(K_Z+B_Z+M_Z)$ is in a fixed finite set (see Lemma \ref{volume fixed}). Thus Theorem \ref{mainthm2} is a consequence of Theorem \ref{mainthm1}.
~\\

\paragraph{\textbf{Acknowledgements}} The author expresses his gratitude to his advisor Meng Chen for great support and encouragement. He would like to thank Jingjun Han for very effective discussions and valuable suggestions. He is grateful to Xiaowei Jiang, Junpeng Jiao, Mengchu Li and Hexu Liu for giving useful comments.  He would also like to thank the anonymous referee for the long list of suggestions and corrections which improved this article considerably.

\section{Preliminaries}

\subsection{Divisors}
\begin{definition}[ACC sets and DCC sets]
	Let $\Phi\subset \bR^{\geq0}$. We say $\Phi$ satisfies the \textit{ascending chain condition} (ACC) if it does not contain an infinite strictly increasing sequence. We say $\Phi$ satisfies the \textit{descending chain condition} (DCC) if it does not contain an infinite strictly decreasing sequence.
\end{definition}

\begin{definition}[Coefficients of divisors]
	Let $\delta\in \bR^{>0}$, $\Phi\subset \bR^{\geq0}$, and $D$ be an $\bR$-divisor on a normal variety $X$. Write $D=\sum a_iD_i$ where $D_i$'s are different Weil divisors on $X$. We denote $D\in \Phi$ if $a_i\in \Phi$ for any $i$, and denote $D\geq \delta$ (resp. $D\leq \delta$) if $a_i\geq\delta$ (resp. $a_i\leq \delta$) for any $i$.
\end{definition}

\begin{definition}[Contractions]
	We say a projective morphism $f:X\to Z$ between normal varieties is a \textit{contraction} if $f_*\cO_X=\cO_Z$. In particular, $f$ has connected fibers.
\end{definition}

\begin{definition}[Horizontal part and vertical part of divisors]
	Let $f:X\to Z$ be a contraction between normal varieties. Let $D$ be an $\bR$-divisor on $X$. We say that $D$ is \textit{vertical} over $Z$ if $f(\Supp D)$ is a proper subset of $Z$. We say that $D$ is \textit{horizontal} over $Z$ if the induced map $\Supp D\to Z$ is dominant. 
	
	Given an $\bR$-divisor $D$ on $X$, there is a unique decomposition $D=D^h+D^v$ such that 
	\begin{itemize}
		\item$\Supp D^h$, $\Supp D^v$ have no common components,
		\item every component of $\Supp D^h$ is horizontal over $Z$, and
		\item $D^v$ is vertical over $Z$.
	\end{itemize}
	We call $D^h$ the \textit{horizontal part} of $D$ and $D^v$ the \textit{vertical part} of $D$ with respect to $f:X\to Z$.
\end{definition}

\begin{definition}[Invariant Iitaka dimensions, {\cite[Definition 2.2.1]{invariantIitakadimension}}]
	Let $D$ be an $\bR$-divisor on a projective normal variety $X$. We define the \textit{invariant Iitaka dimension} $\kappa_\iota(X,D)$ as follows.  If $|D|_\bR \neq \emptyset$, let $\kappa_\iota(X,D)=\kappa(X,\DD)$ for some $\bR$-divisor $\DD\in |D|_\bR$. Here, the right hand side is the usual Iitaka dimension of $\DD$.  Note that in this case $\kappa_\iota(X,D)$ does not depend on the choice of $\DD$ by \cite[Corollary 2.1.4]{invariantIitakadimension}. If $|N|_\bR=\emptyset$, let $\kappa_\iota(X,D)=-\infty$.
\end{definition}

Next we generalize the definition of Iitaka volume in \cite{Iitakavolume} to $\bR$-divisors. It is also called pseudo-volume in \cite{invariantIitakadimension}.
\begin{definition}\label{Iitaka volume}(Iitaka volumes of $\bR$-divisors, {\cite[Definition 1.1]{Iitakavolume}, see also \cite [Definition 2.1.1]{invariantIitakadimension}})
	Let $D$ be an $\bR$-divisor on a projective normal variety $X$ with invariant Iitaka dimension $\kappa_\iota(D)$. We define \textit{Iitaka volume} $\Ivol(D)$ of $D$ as follows. If $\kappa_\iota(D)\geq 0$, let $\DD$ be an element of $|D|_\bR$, and then
	$$\Ivol(D):=\limsup_{m\to \infty} \frac{h^0(\rounddown{m\DD})}{m^{\kappa_\iota(D)}/\kappa_\iota(D)!}.$$ Note that in this case $\Ivol(D)$ does not depend on the choice of $\DD$ by \cite[Corollary 2.1.4]{invariantIitakadimension}.
	If $\kappa_\iota(D)=-\infty$, then let $\Ivol(D)=0$.
\end{definition}

If $f: X\to Z$ is a contraction between two normal varieties and $D\sim_\bR f^*L$ for some big $\bR$-divisor $L$ on $Z$, then $\Ivol(D)=\vol(L)$. 

\begin{definition}[b-divisors]
	Let $X$ be a normal variety. A \textit{b-divisor} $\textbf{M}$ is a collection of $\bR$-divisors $M_Y$ on $Y$ for each birational contraction $Y\to X$ from a normal variety and satisfies the following: if $\YY\to Y\to X$ are birational contractions, then the pushdown of $M_{\YY}$ on $Y$ is $M_Y$.
	
	We say a b-divisor $\textbf{M}$ is \textit{b-$\bR$-Cartier} if there is a birational contraction $Y\to X$ such that 
	\begin{itemize}
		\item $M_Y$ is $\bR$-Cartier, and
		\item if $\YY\to Y$ is a birational contraction, then $M_{\YY}$ is the pullback of $M_Y$.
	\end{itemize}
	In this case, we say that the b-$\bR$-Cartier divisor $\textbf{M}$ descends on $Y$ and is represented by $M_Y$. Note that the representation is  not unique, if $\YY \to X$ is another birational contraction and $M_{\YY}$ is an $\bR$-Cartier divisor on $\YY$, then $M_Y$ and $M_{\YY}$ define the same b-$\bR$-Cartier b-divisor if the pullbacks of $M_Y$ and $M_{\YY}$ to a common resolution of $Y$ and $\YY$ are the same.
\end{definition}

\subsection{(Generalized) Pairs and Singularities}
\begin{definition}[Pairs and Singularities]
	Let $X$ be a normal quasi-projective variety and $B$ be an $\bR$-divisor on $X$. We say that $(X,B)$ is a \textit{sub-pair} if $K_X+B$ is $\bR$-Cartier. If in addition $B\geq 0$, then $(X,B)$ is a \textit{pair}.
	
	Let $D$ be a prime divisor over $X$, i.e. there is a birational model over $X$ such that $D$ is a prime divisor on this model. Let $W\to X$ be a log resolution of a sub-pair $(X,B)$ so that $D$ is a prime divisor on $W$. Let $K_W+B_W$ be the pullback of $K_X+B$. Define the \textit{log discrepancy} of the prime divisor $D$ as $1-\mu_D B_W$, where $\mu_D B_W$ means the coefficient of $D$ in $B_W$. We denote the log discrepancy of $D$ with respect to $(X,B)$ as $a(D,X,B)$.
	
	We say that a sub-pair $(X,B)$ is \textit{sub-klt} (resp. \textit{sub-lc}, \textit{sub-$\epsilon$-lc}) if $a(D,X,B)>0$ (resp. $a(D,X,B)\geq0$, $a(D,X,B)\geq \epsilon$) for every prime divisor $D$ over $X$. If $(X,B)$ is a pair, then we remove the sub and say the pair is klt (resp. lc, $\epsilon$-lc).
	
	Let $(X,B)$ be a sub-pair. A \textit{non-klt place} (resp. \textit{non-lc place}) is a prime divisor $D$ over $X$ such that $a(D,X,B)\leq 0$ (resp. $a(D,X,B)<0$). A \textit{non-klt center} (resp. \textit{non-lc center}) is the image of a non-klt place (resp. non-lc place).  The \textit{non-klt locus} (resp. \textit{non-lc locus}) of $(X,B)$ is the union of all non-klt places (resp. non-lc places) of $(X,B)$ and denoted as $\Nklt(X,B)$ (resp. $\Nlc(X,B)$).
\end{definition}

\begin{definition}[Generalized pairs and Singularities, {\cite[Definition 1.4, Definition 4.1]{effectiveIitaka}}]
	A \textit{generalized sub-pair} consists of
	\begin{itemize}
		\item a normal variety $X$ equipped with a projective morphism $X\to Z$,
		\item an $\bR$-divisor $B$ on $X$, and
		\item a b-$\bR$-Cartier b-divisor over $X$, represented by a projective birational morphism $f: \XX\to X$ and an $\bR$-Cartier $\bR$-divisor $\MM$ on $\XX$
	\end{itemize}
	such that $\MM$ is nef over $Z$ and $K_X+B+M$ is $\bR$-Cartier, where $M:=f_*\MM$. If in addition $B\geq 0$, then $(X,B+M)$ is a \textit{generalized pair}. Since a b-$\bR$-Cartier b-divisor is defined birationally, in practice we will often replace $\XX$ with a higher model and replace $\MM$ with its pullback.  In this article, we omit $Z$ but say the generalized pair is projective when $Z$ is a point.
 
	Let $D$ be a prime divisor over $X$. Replace $\XX$ with a log resolution of $(X,B)$ such that $D$ is a prime divisor on $\XX$. We can write 
	$$K_{\XX}+\BB+\MM=\pi^*(K_X+B+M).$$ 
	Then we define the \textit{generalized log discrepancy} of  $D$ to be $a(D,X,B+M)=1-\mu_D \BB$.
	
	We say that $(X,B+M)$ is \textit{generalized klt} (resp. \textit{generalized lc}, \textit{generalized $\epsilon$-lc}) if $a(D,X,B+M)>0$ (resp. $a(D,X,B+M)\geq0$, $a(D,X,B+M)\geq \epsilon$) for every prime divisor $D$ over $X$. 
\end{definition}

\subsection{Adjunction formulas for fiber spaces}
We recall the construction of adjunction formulas for fiber spaces based on \cite{kawamataadjunctionformula, ambroadjunctionformula,ambromoduli}.
Let $(X,B)$ be a projective sub-pair and let $f:X\to Z$ be a contraction between quasi-projective normal varieties with $\dim Z>0$ such that $(X,B)$ is sub-lc near the generic fiber of $f$ and $K_X+B\sim_{\bR}0/Z$. 

Fix a prime divisor $D$ on $Z$ and let $t_D$ be the lc threshold of $f^*D$ with respect to $(X,B)$ over the generic point of $D$, i.e. $t_D$ is the largest number so that $(X,B+t_Df^*D)$ is sub-lc over the generic point of $D$. Now let $b_D=1-t_D$ and by basic argument there are finitely many prime divisors $\DD$ on $Z$ such that $b_\DD\neq0$. Hence we can define $B_Z=\sum b_D D$, where the sum runs over all the prime divisors on $Z$. 

Since $K_X+B\sim_{\bR}0/Z$, there is an $\bR$-Cartier $\bR$-divisor $L_Z$ on $Z$ such that $K_X+B\sim_{\bR}f^*L_Z$. Let $M_Z=L_Z-(K_Z+B_Z)$ and we have the following \textit{adjunction formula} $$K_X+B\sim_{\bR} f^*(K_Z+B_Z+M_Z).$$ We call $B_Z$ the \textit{discriminant divisor} and $M_Z$ the \textit{moduli divisor} of $(X,B)$ with respect to $f:X\to Z$. Note that $B_Z$ is uniquely determined but $M_Z$ is determined only up to $\bR$-linear equivalence. 

Take a commutative diagram 
\begin{displaymath}
	\xymatrix{
		\XX \ar[d]_{\ff} \ar[r]^\pi  & X \ar[d]^f         \\
		\ZZ \ar[r]^\mu  & Z   }
\end{displaymath}
such that $\mu$ and $\pi$ are birational contractions. Let $K_{\XX}+\BB$ be the pullback of $K_X+B$ on $\XX$ and similarly we can define a discriminant divisor $B_{\ZZ}$ and $L_{\ZZ}=\mu^*L_Z$ gives a moduli divisor $M_{\ZZ}$ so that 
$$K_{\XX}+\BB\sim_{\bR}\ff^*(K_{\ZZ}+B_{\ZZ}+M_{\ZZ}).$$ It is easy to see that $B_Z$ is the pushdown of $B_{\ZZ}$ and $M_Z$ is the pushdown of $M_{\ZZ}$. Therefore, $B_Z$ and $M_Z$ can be regarded as b-divisors.

The following lemma shows that when $(X,B)$ is lc over the generic point of $Z$, $(Z,B_Z+M_Z)$ is a generalized pair.

\begin{lem}{\rm(Adjunction formula for $\bR$-coefficients, \cite[Theorem 2.23]{JLX22}})\label{adjunction formula for R coefficients}
	With the above notation and assumptions, suppose that $(X,B)$ is lc over the generic point of $Z$. Then $M_{\ZZ}$ is nef on some high resolution $\ZZ\to Z$, and $(X,B+M)$ is a generalized pair.
\end{lem}

Here we recall the M\textsuperscript{c}Kernan-Shokurov type conjecture which was proved by Birkar recently. 
\begin{lem}{\rm(\cite[Theorem 1.8]{singularitiesonFanofibrations})}\label{singularities on base}
	Let $d\in \bN$, and $u,v,\epsilon\in \bR^{>0}$. Then there exists $\delta\in \bR^{>0}$ depending only on $d,u,v,\epsilon$ satisfying the following. 
	
	Assume that 
	\begin{itemize}
		\item $(X,B)$ is an $\epsilon$-lc pair,
		\item $f: X\to Z$ is a contraction and $\dim X-\dim Z\leq d$,
		\item $K_X+B\sim_{\bR}0/Z$,
		\item $A$ is an effective $\bR$-divisor on $X$ such that $A\geq u$, and
		\item $0<\vol(A|_F)<v$ for the general fibers $F$ of $f$.
	\end{itemize}
	Then the generalized pair $(Z,B_Z+M_Z)$ given by the adjunction formula
	\begin{equation*}
		K_X+B\sim_{\bR} f^*(K_Z+B_Z+M_Z)
	\end{equation*}
	is generalized $\delta$-lc.
\end{lem}

\begin{proof}
	Since $A\geq u$, we conclude that $$0<\vol(A_{\red}|_F)\leq \vol(\frac{1}{u}A|_F)\leq \frac{v}{u^{\dim F}}$$ where $F$ is a general fiber of $f$.
	
	Using approximation, we can write $B=\sum_{i=1}^{l}r_iB_i$ such that 
	\begin{itemize}
		\item $\sum_{i=1}^{l}r_i=1$ and $r_i\geq 0$ for any $i$,
		\item $r_i$'s are $\bQ$-linear independent, and
		\item $(X,B_i)$ is an $\frac{\epsilon}{2}$-lc $\bQ$-pair for any $i$.
	\end{itemize}
	Moreover, we have $K_X+B_i\sim_{\bQ}0/Z$ by \cite[Lemma 5.3]{HLS19}. 
	
	Let $D$ be a prime divisor over $Z$.  Take a commutative diagram 
	\begin{displaymath}
		\xymatrix{
			\XX \ar[d]_{\ff} \ar[r]^\pi  & X \ar[d]^f         \\
			\ZZ \ar[r]^\mu  & Z   }
	\end{displaymath}
	such that $\mu$ and $\pi$ are birational contractions and $D$ is a prime divisor on $\ZZ$. Write $K_{\XX}+\BB=\pi^*(K_X+B)$ and $K_{\XX}+\BBi=\pi^*(K_X+B_i)$. Let $t_D$ (resp. $t_{i,D}$) be the lc threshold of $\ff^*D$ with respect to $(\XX,\BB)$ (resp. $(\XX,\BBi)$) over the generic point $\eta_D$ of $D$. Since $(\XX,\BBi+t_{i,D}\ff^*D)$ is sub-lc over $\eta_D$, $(\XX,\BB+\sum_{i=1}^l r_it_{i,D}\ff^*D)$ is also sub-lc over $\eta_D$, hence $t_D\geq \sum_{i=1}^l r_it_{i,D}$.
	
	Applying \cite[Theorem 1.8]{singularitiesonFanofibrations} to $(X, B_i), A_{\red}$, we deduce that there exists $\delta\in \bR^{>0}$ depending only on $d,u,v,\epsilon$ such that $t_{i,D}\geq \delta$. Therefore, $$t_D\geq \sum_{i=1}^lr_it_{i,D}\geq \sum_{i=1}^l r_i \delta=\delta.$$ Hence the discriminant b-divisor $\mathbf{B}_Z$ has coefficients $\leq 1-\delta$ and $(Z, B_Z+M_Z)$ is generalized $\delta$-lc.
\end{proof}

\subsection{Bounded families}
\begin{definition}[Bounded families of couples and pairs]
	A \textit{couple} consists of a projective normal variety $X$ and a reduced divisor $D$ on $X$. We say that two couples $(X,D)$ and $(\XX,\DD)$ are isomorphic if there is an isomorphism $X\to \XX$ mapping $D$ onto $\DD$. 
	
	Let $\cP$ be a set of couples. Assume that 
	\begin{itemize}
		\item there exist finitely many projective morphisms $V^i\to T^i$ of varieties,
		\item $C^i$ is a reduced divisor on $V^i$, and
		\item  for each $(X,D)\in \cP$ there exists an $i$, a closed point $t\in T^i$ and an isomorphism $\phi: V^i_t\to X$ such that $(V^i_t,C^i_t)$ is a couple and $\phi_*C^i_t\geq D$. 
	\end{itemize}
	Then we say that $\cP$ is \textit{bounded}. This is equivalent to say that there is a positive integer $r$ such that for each $(X,D)\in \cP$, we can find a very ample divisor $A$ on $X$ such that $A^{\dim X}\leq  r$ and $D\cdot A^{\dim X-1}\leq r$ (cf. \cite[Lemma 2.20]{BABI}).
	
	A set of projective pairs $(X,B)$ is said to be bounded if the set of $(X,\Supp B)$ forms a bounded family of couples.
\end{definition}

\subsection{Families of generalized pairs}
\begin{definition}[{\cite[Definition 1.1]{BVGP}}]\label{familyofgpairs}
	Let $d\in \bN$ ,$\Phi\subset \bR^{\geq 0}$, and $v\in \bR^{>0}$.
	\begin{enumerate}[itemsep=13pt]
		\item Let $\cF_{gklt}(d,\Phi)$ be the set of projective generalized pairs $(X,B+M)$ with data $\XX\to X$ and $\MM$ such that 
		\begin{itemize}
			\item $(X, B+M)$ is generalized klt of dimension $d$,
			\item $B\in \Phi$,
			\item $\MM=\sum \mu_i\MMi$ where $\MMi$ is Cartier nef and $\mu_i\in \Phi$ for any $i$, and
			\item $K_X+B+M$ is ample.
		\end{itemize}
		
		\item Let $$\cF_{gklt}(d,\Phi,v)\subseteq \cF_{gklt}(d,\Phi)$$ consist of those $(X,B+M)$ such that $\vol(K_X+B+M)=v$. Similarly, let $$\cF_{gklt}(d,\Phi,\leq v)\subseteq \cF_{gklt}(d,\Phi)$$ consist of those $(X,B+M)$ such that $\vol(K_X+B+M)\leq v$.
		
\end{enumerate}
\end{definition}

Here we give a lemma to show that if $(X,B+M)\in \cF_{gklt}(d,\Phi,v)$, then we can control the Cartier index of any $\bQ$-Cartier Weil divisor on $X$.
\begin{lem}\label{bound Cartier index}
	Let $d\in \bN$, $v\in \bR^{>0}$, and $\Phi\subset \bR^{\geq0}$ be a DCC set. Then there exists $N\in \bN$ depending only on $d,\Phi,v$ such that for any $(X,B+M)\in \cF_{gklt}(d,\Phi,v)$ and any $\bQ$-Cartier Weil divisor $D$ on $X$, the Cartier index of $D$ divides $N$.
\end{lem}

\begin{proof}
	Let $(X,B+M)\in \cF_{gklt}(d,\Phi,v)$ and $D$ be a $\bQ$-Cartier Weil divisor on $X$. By Step 7 of the proof of \cite[Theorem 1.4]{BVGP}, there is a boundary $\Theta$ on $X$ such that $(X,\Theta)$ is $\epsilon$-lc for some positive real number $\epsilon$ depending only on $d,\Phi,v$, and  $(X,\Theta)$ belongs to a log bounded family $\cP$. By \cite[Theorem 1.10]{HLQ23}, There is a positive integer $N$ depending only on $\epsilon,\cP$, hence depending only on $d,\Phi,v$, such that the Cartier index of $D$ divides $N$.
\end{proof}

\subsection{Decomposition of $\bR$-coefficient generalized pairs in $\cF_{gklt}(d,\Phi,v)$}
In this subsection, we decompose an $\bR$-coefficient generalized pair $(X,B+M)\in\cF_{gklt}(d,\Phi,v)$ into $\bQ$-coefficient generalized pairs $(X,B_i+M_i)$ such that a bounded multiple of $K_X+B_i+M_i$ is very ample.
\begin{lem}\label{effective very ample}
	Let $d\in \bN$. Then there exists $m\in \bN$ depending only on $d$ satisfying the following. Assume that 
	\begin{itemize}
		\item $(X,B+M)$ is a generalized klt pair of dimension $d$,
		\item $L$ is an ample Cartier divisor on $X$, and
		\item $L-K_X-B-M $ is nef and big.
	\end{itemize}     
	Then $mL$ is very ample.
\end{lem}

\begin{proof}
	By the proof of \cite[Lemma 2.4]{Contractiontheoremforgeneralizedpairs}, there exists an effective $\bR$-divisor $\Delta$ such that $(X,\Delta)$ is klt and $L-K_X-\Delta$ is ample. Then by effective base point free theorem \cite[Theorem 1.1]{effectivebasepointfree}, $|mL|$ is base point free for some positive integer $m$ depending only on $d$. By effective very ampleness lemma \cite[Lemma 7.1]{fujinoeffective}, replacing $m$ with a bounded multiple, we conclude that $mL$ is very ample.
\end{proof}

\begin{thm}\label{decomposition of ample divisor}
	Let $d\in \bN$, $v\in \bR^{>0}$, and $I\subset \bR^{\geq 0}$ be a finite set. Then there exists a finite set $J\subset \bR^{>0}$ and $p\in \bN$ depending only on $d,I,v$ satisfying the following.
	
	If $(X,B+M)\in \cF_{gklt}(d,I,v)$, then we can decompose $K_X+B+M$ as follows:
	\begin{equation*}
		K_X+B+M=\sum_{i=1}^l r_i(K_X+B_i+M_i)
	\end{equation*} 
	such that 
	\begin{itemize}
		\item $\sum_{i=1}^l r_i=1$ and $r_i\in J$ for any $i$,
		\item $(X,B_i+M_i)$ is a generalized klt pair with nef part $\MMi$ on some high resolution $\XX\to X$ for any $i$, and 
		\item $p(K_X+B_i+M_i)$ is very ample and $p\MMi$ is Cartier nef for any $i$.
	\end{itemize}
\end{thm}

\begin{proof}
	Pick a generalized pair $(X,B+M)\in  \cF_{gklt}(d, I,v)$ with data $\XX\to X$ and $\MM$. By \cite[Theorem 3.15]{CGD20} we can write
	$$K_X+B+M=\sum_{i=1}^l r_i(K_X+B_i+M_i)$$ such that
	\begin{itemize}
		\item $\sum_{i=1}^l r_i=1$ and $r_i\in J$, where $J\subset \bR^{> 0}$ is a finite set depending only on  $d, I$,
		\item $(X,B_i+M_i)$ is a generalized klt pair with nef part $\MMi$ on $\XX$ for any $i$, 
		\item $\MM=\sum_{i=1}^l r_i \MMi$, and
		\item there exists a positive integer $p$ depending only on $d, I$ such that $p(K_X+B_i+M_i)$ is integral, and $p\MMi$ is Cartier nef for any $i$.
	\end{itemize}
	Note that $l\leq N:=\frac{1}{\min J}$, because $\sum_{i=1}^l r_i=1$. By Lemma \ref{bound Cartier index}, replacing $p$ with a bounded multiple we can assume that $p(K_X+B_i+M_i)$ is Cartier. Now $p$ depends on $d, I,v$.
	
	Consider the set $$\Lambda:=\{\sum_{i=1}^h \frac{\nu_in_i}{p} | n_i\geq -2dp, n_i\in \bZ, \nu_i\in J, h\leq N\}.$$ It is easy to see that $\Lambda$ is a DCC set, hence we can take $\delta:=\min \{\alpha>0| \alpha\in \Lambda\}$. 
	
	If $K_X+B_i+M_i$ is not nef for some $i$, let $R$ be a $(K_X+B_i+M_i)$-negative extremal ray. Let $C$ be an extremal curve of $R$, which means that there exists an ample divisor $H$ such that $$H\cdot C=\min\{H\cdot \Gamma|\;[\Gamma]\in R\}.$$ By the length of extremal ray for generalized pairs \cite[Proposition 3.13, Lemma 3.5]{weakzariskidecomposition}, $$(K_X+B_i+M_i)\cdot C\geq -2d$$ for $1\leq i\leq l$. Since $p(K_X+B_i+M_i)$ is Cartier, we deduce that $p(K_X+B_i+M_i)\cdot C$ is an integer in $[-2dp,\infty)$. Since $K_X+B+M$ is ample and
	\begin{align*}
		(K_X+B+M)\cdot C&=\sum_{i=1}^l r_i(K_X+B_i+M_i)\cdot C\\
		&=\sum_{i=1}^l \frac{r_i}{p} (p(K_X+B_i+M_i)\cdot C)\in \Lambda,
	\end{align*}
	hence $(K_X+B+M)\cdot C\geq \delta$. Therefore,
	$$\left( \frac{\delta}{2d+\delta}(K_X+B_i+M_i)+\frac{2d}{2d+\delta}(K_X+B+M) \right) \cdot C\geq 0.$$
	This argument implies that  $$\frac{\delta}{2d+\delta}(K_X+B_i+M_i)+\frac{2d}{2d+\delta}(K_X+B+M)$$ is nef.
	Now we take $$K_X+\tilde{B_i}+\tilde{M_i}:=\frac{4d+\delta}{4d+2\delta}(K_X+B+M)+\frac{\delta}{4d+2\delta}(K_X+B_i+M_i),$$ then $$K_X+B+M=\sum_{i=1}^l r_i(K_X+\tilde{B_i}+\tilde{M_i}).$$
	 Since $K_X+B+M$ is ample, we conclude that $K_X+\tilde{B_i}+\tilde{M_i}$ is also ample. Consider the convex hull $\mathcal{H}$ spanned by $K_X+\tilde{B_i}+\tilde{M_i}$, i.e. $$\mathcal{H}:=\{\sum_{i=1}^l \lambda_i(K_X+\tilde{B_i}+\tilde{M_i})| \lambda_i\geq 0, \sum_{i=1}^l\lambda_i=1\}.$$ Since $K_X+B+M$ is an interior point of $\mathcal{H}$, we can choose $K_X+\overline{B}_i+\overline{M}_i\in \mathcal{H}$ such that
	\begin{itemize}
		\item $K_X+B+M=\sum_{i=1}^l \overline{r}_i(K_X+\overline{B}_i+\overline{M}_i)$, where $\overline{r}_i$ belongs to a finite set $\overline{J}\subset\bR^{>0}$ depending only on $J, d, \delta$, hence depending only on $d,I,v$,
		\item $(X,\overline{B}_i+\overline{M}_i)$ is a generalized klt pair with nef part $\overline{M}^{\prime}_i$ on $\XX$ for each $i$,
		\item $(K_X+\overline{B}_i+\overline{M}_i)$ is ample for each $i$, and
		\item there exists a positive integer $\overline{p}$ depending only on $J, d, \delta, p$, hence depending only on $d,I,v$ such that $\overline{p}(K_X+\overline{B}_i+\overline{M}_i)$ is integral and $\overline{p}\overline{M}^{\prime}_i$ is Cartier nef for any $i$.
	\end{itemize}
	By Lemma \ref{bound Cartier index}, replacing $\overline{p}$ with a bounded multiple we can assume that $\overline{p}(K_X+\overline{B}_i+\overline{M}_i)$ is Cartier. By Lemma \ref{effective very ample}, replacing $\overline{p}$ again with a bounded multiple we can assume that $\overline{p}(K_X+\overline{B}_i+\overline{M}_i)$ is very ample.  Now replace $r_i$ with $\overline{r}_i$, $J$ with $\overline{J}$, $(X,B_i+M_i)$ with $(X,\overline{B}_i+\overline{M}_i)$, and $p$ with $\overline{p}$, and then we finish the proof.
\end{proof}

\subsection{Relative nefness}
\begin{definition}\label{relatively nefness}
	Let $s, k$ be integers such that $0\leq k\leq s$. Let $f:X\to Z$ be a contraction between projective normal varieties with $\dim Z=s$. Let $(X,B)$ be a pair.  We say $K_X+B$ is \textit{nef in dimension $k$ over $Z$}, if for any  very ample divisors $H_1, \cdots, H_{s-k}$ on $Z$, the following is satisfied.
	
	Let $L_i$ be  a general member of $|H_i|$ and $R_i=f^{-1}L_i$. Take $S=\cap_{i=1}^{s-k}R_i$ and write $$K_S+B_S=(K_X+B+\sum_{i=1}^{s-k}R_i)|_S,$$  then $K_S+B_S$ is nef. 	
	
	Note that if $K_X+B$ is nef over $Z$, then $K_X+B$ is nef in dimension $0$ over $Z$.
\end{definition}

The following lemma will be used in the proof of Theorem \ref{mainthm1}. 

\begin{lem}\label{nef in dimension}
	Let $d, s, k$ be integers such that $0\leq k\leq s\leq d$. Assume that 
	\begin{itemize}
		\item  $f: X \to Z$ is a contraction between projective normal varieties with $\dim X=d$ and $\dim Z=s$,
		\item  $(X,B)$ is a pair and $K_X+B$ is nef in dimension $k$ over $Z$,
		\item $H_1, \cdots,H_{s-k}$ are very ample divisors on $Z$, and
		\item  $N$ is a nef $\bR$-divisor on $X$.
	\end{itemize}
	Then $$(K_X+B+N+\sum_{i=1}^{s-k}f^*H_i)^{d-s+k}\cdot f^*H_1 \cdot \ldots \cdot f^*H_{s-k}\geq0.$$
\end{lem}

\begin{proof}
	By definition, if we take $L_i$ as a general member of $|H_i|$, $R_i=f^{-1}L_i$ and $S=\cap_{i=1}^{s-k}R_i$, then $$K_S+B_S=(K_X+B+\sum_{i=1}^{s-k}R_i)|_S$$ is nef, hence $K_S+B_S+N_S$ is nef, where $N_S=N|_S$. Therefore,
	\begin{align*}
		&(K_X+B+N+\sum_{i=1}^{s-k}f^*H_i)^{d-s+k}\cdot f^*H_1 \cdot \ldots \cdot f^*H_{s-k}\\
		=&(K_X+B+N+\sum_{i=1}^{s-k}f^*L_i)^{d-s+k}\cdot f^*L_1 \cdot \ldots \cdot f^*L_{s-k}\\
		=&(K_X+B+N+\sum_{i=1}^{s-k}R_i)^{d-s+k}\cdot R_1 \cdot \ldots \cdot R_{s-k}\\
		=&(K_S+B_S+N_S)^{d-s+k}\geq 0.
	\end{align*}
\end{proof}

\section{Effective adjunction formula with real coefficients}
In this section we extend the effective adjunction formula \cite[Lemma 7.4]{BVGP} to the real coefficients case. The main new tool is uniform rational polytopes for canonical bundle formulas developed in \cite{effectiveadjunctionformulawithrealcoefficients}. The effective adjunction formula is one of the main ingredients in the proof of Theorem \ref{mainthm1} and Theorem \ref{singularities of klt stable minimal models}.

\subsection{Cartier index of moduli divisors} Given $q\in \bN$ and two $\bR$-divisors $C$, $D$ on a normal variety $X$, if $qC\sim qD$, then we write $C\sim_q D$.

\begin{lem}\label{moduli part Cartier}
	Let $d, q\in \bN$, $u\in \bR^{>0}$ and $\Phi\subset \bR^{\geq0}$ be a DCC set. Then there exists $p\in \bN$ depending only on $d, q, u, \Phi$ satisfying the following. Assume that 
	\begin{itemize}
		\item $(X,B)$ is a projective lc pair of dimension $d$,
		\item $f: X \to Z$ is a contraction with $K_X+B\sim_{\bQ}0/Z$,
		\item we have an adjunction formula
		\begin{equation*}
			K_X+B\sim_q f^*(K_Z+B_Z+M_Z),
		\end{equation*}
		\item $A\in \Phi$ is an effective $\bR$-divisor on $X$,
		\item over some non-empty open subset $U\subseteq Z$: $(X,B+tA)$ is lc for some $t>0$, and $A$ is relatively semi-ample, and 
		\item for the general fiber $F$ of $f$, we have $0<\vol(A|_F)\leq u$.
	\end{itemize} 
	Then $pM_{\ZZ}$ is Cartier on some high resolution $\ZZ\to Z$.
\end{lem}

\begin{proof}
	We follow the proof of \cite[Lemma 7.4]{BVGP}.
	\begin{enumerate} [label=\textsl{Step} \arabic{enumi}., wide=13pt, itemsep=13pt]
		\item Take a resolution $\ZZ\to Z$ so that $\textbf{M}$ descends to $\ZZ$, and a log resolution $W\to X$ of $(X,B)$ such that $W\dasharrow \ZZ$ is a morphism. Let $\Delta_W$ be the sum of the horizontal/$\ZZ$ part of reduced exceptional divisors and the birational transform of the horizontal/$Z$ part $B^h$ of $B$. After finitely many blow ups and changing $\Delta_W$ accordingly, we can assume that every non-klt center of $(W,\Delta_W)$ is horizontal over $\ZZ$.
		
		Run an MMP on $K_W+\Delta_W$ over $X$ with scaling of some ample divisor. Since over the generic point $\eta_Z$ of $Z$, $K_W+\Delta_W$ is the sum of the pullback of $K_X+B^h$ and an effective exceptional divisor, the MMP terminates over $\eta_Z$ by \cite[Theorem 1.8]{Bir12}. Thus we get a model $(V,\Delta_V)$ on which $K_V+\Delta_V \sim_{\bQ}0$ over $\eta_Z$, and hence over the generic point $\eta_{\ZZ}$ of $\ZZ$. Applying \cite[Theorem 1.4]{Bir12} or \cite[Theorem 1.1]{HX13}, we can run an MMP on $K_W+\Delta_W$ over $\ZZ$ ending with a good minimal model. Replacing $V$ with the good minimal model of $(W,\Delta_W)$ over $\ZZ$, we can assume that $K_V+\Delta_V$ is semi-ample over $\ZZ$ and induces a contraction $f^{\prime}: V\to \ZZZ /\ZZ$. 
		
		\item  Take a common resolution $\pi: Y\to X$ and $\pi^\prime: Y\to V$. Then 
		\begin{equation*}
			P_Y:= \pi^{\prime *}(K_V+\Delta_V)-\pi^*(K_X+B)
		\end{equation*}
		is vertical over $\ZZZ$. In addition, since $K_X+B\sim_{\bQ}0/Z$ and $K_V+\Delta_V\sim_{\bQ}0/\ZZZ$, we conclude that $P_Y\sim_{\bQ}0/\ZZZ$, which implies that $P_Y$ is the pullback of an $\bQ$-Cartier $\bQ$-divisor $P_{\ZZZ}$ by \cite[Lemma 2.5]{CHL23}. The adjunction formula 
		\begin{equation*}
			K_X+B\sim_q f^*(K_Z+B_Z+M_Z)
		\end{equation*}
		induces the following adjunction formula 
		\begin{equation*}
			(\ast) \quad K_V+\Delta_V\sim_q f^{\prime*}(K_{\ZZZ}+B_{\ZZZ}+P_{\ZZZ}+M_{\ZZZ}),
		\end{equation*}
		where $K_{\ZZZ}+B_{\ZZZ}+M_{\ZZZ}$ is the pullback of $K_Z+B_Z+M_Z$ on $\ZZZ$.  From $(\ast)$ we see that the discriminant divisor $\Delta_{\ZZZ}$ of $(V,\Delta_V)$ over $\ZZZ$ is $B_{\ZZZ}+P_{\ZZZ}$ and the moduli divisor of $(X,B)\to Z$ on $\ZZZ$ coincides with the moduli divisor of $(V,\Delta_V)\to \ZZZ$. 
		
		\item We claim that there exists $p\in \bN$ depending only on $d, q,  u, \Phi$ such that $pM_{\ZZZ}$ is integral. Then $pM_{\ZZ}$ is integral and hence Cartier because $\ZZ$ is smooth.
		
		 As in the fifth paragraph of the proof of \cite[Lemma 7.4]{BVGP}, we can reduce to the case where $Z$ and $\ZZZ$ are curves. Then the claim follows from \cite[Lemma 7.3]{BVGP}. Note that in \cite[Lemma 7.3]{BVGP}, $A$ is an effective integral divisor on $X$. In our case, the coefficients of $A$ are in a DCC set of positive real numbers.  The proof of \cite[Lemma 7.3]{BVGP} can still be applied with some slight changes. In Step 2 of the proof of \cite[Lemma 7.3]{BVGP}, replace \cite[Theorem 1.1]{HX13} with \cite[Theorem 1.2]{Hashizume19}, which is an $\bR$-version of \cite[Theorem 1.1]{HX13}, and in Step 5, replace \cite[Theorem 1.7]{GOPV} with \cite[Theorem 6.4]{GOPV}.		
	\end{enumerate}
\end{proof}

\subsection{Finite coefficients case}

We first give an effective adjunction formula in the case that the coefficients of $B$ are in a fixed finite set of real numbers.
\begin{thm}\label{finite effective adjunction}
	Let $d\in \bN$, $u\in \bR^{>0}$, $I\subset \bR^{\geq0}$ be a finite set, and $\Phi\subset \bR^{\geq0}$ be a DCC set. Then there exists a finite set $J \subset \bR^{\geq0}$ depending only on  $d, u, I, \Phi$ satisfying the following. Assume that
	\begin{itemize}
		\item $(X,B)$ is  a projective lc pair of dimension $d$ and $B\in I$,
		\item $f: X \to Z$ is a contraction with $K_X+B\sim_{\bR} 0/Z$,
		\item $(X,B)$ is klt over the generic point $\eta_Z$ of $Z$,
		\item $A\in \Phi$ is an effective $\bR$-divisor on $X$ such that  $A$ is relatively semi-ample over the generic point $\eta_Z$ of $Z$, and
		\item $0<\vol(A|_F)\leq u$ for the general fiber $F$ of $f$. 
	\end{itemize} 
	
	Then there exists an adjunction formula 
	\begin{equation*}
		K_X+B \sim_{\bR} f^*(K_Z+B_Z+M_Z)
	\end{equation*}
	such that $M_{\ZZ}=\Sigma_{i=1}^l r_iM_{i,\ZZ}$ on some high resolution $\ZZ\to Z$, where $r_i\in J$ and  $M_{i,\ZZ}$ is Cartier nef for any $i$. 
	
	Moreover, if in addition $(X,B)$ is $\epsilon$-lc for some $\epsilon\in \bR^{>0}$, then the coefficients of $B$ are in a fixed finite set $\Psi$ depending only on $d,u,\epsilon,I,\Phi$.
\end{thm}

\begin{proof}
	\begin{enumerate} [label=\textsl{Step} \arabic{enumi}., wide=13pt, itemsep=13pt]
		\item By \cite[Theorem 3.3]{effectiveadjunctionformulawithrealcoefficients}, we can write $$K_X+B=\sum_{i=1}^l r_i(K_X+B_i)$$ such that 
		\begin{itemize}
			\item $r_1,\cdots,r_l$ belong to a finite set $J\subset \bR^{\geq0}$ depending only on $d,I$,
			\item $r_1,\cdots,r_l$ are $\bQ$-linear independent and $\sum_{i=1}^l r_i=1$,
			\item $(X,B_i)$ is lc and $\Nklt(X,B_i)=\Nklt(X,B)$ for any $i$ (in particular, $(X,B_i)$ is klt over the generic point $\eta_Z$ of $Z$ for any $i$), 
			\item $K_X+B_i\sim_{\bQ}0/Z$ for any $i$,
			\item there exists $q\in \bN$ depending only on $d,I$ such that $q(K_X+B_i)$ is integral for any $i$, and
			\item if $\textbf{M}$ and $\textbf{M}_i$ are the moduli part of the adjunction formulas with respect to $(X,B)$ and $(X,B_i)$, then $\textbf{M}=\sum_{i=1}^l r_i \textbf{M}_i$.
		\end{itemize}

		Take $F$ as a general fiber of $f$, then $K_F+B_{i,F}=(K_X+B_i)|_F\sim_{\bQ}0$. Hence $(F,B_{i,F}),A_F$ is a polarized klt Calabi-Yau pair. Since the coefficients of $B_{i,F}$ are in a finite set and the coefficients of $A_F$ are in a DCC set $\Phi$, they are bounded from below away from zero. By \cite[Lemma 2.48]{BABI}, $(F,B_{i,F})$ is $\delta$-lc for some positive real number $\delta$ depending only on $d,I$. Therefore, $(F, \Supp(B_{i,F}+A_F))$ belongs to a bounded family depending only on $\delta, \Phi, u$ by \cite[Theorem 6.2]{GOPV}. Hence by \cite[Lemma 7.2]{BVGP}, possibly replacing $q$ with a bounded multiple, we can assume that $q(K_F+B_{i,F})\sim 0$.  This implies that we can find a rational function $\alpha_i$ on $X$ such that $q(K_X+B_i)+\Div(\alpha_i)$ is vertical over $Z_i$. Since $$q(K_X+B_i)+\Div(\alpha_i)\sim_{\bQ}0/Z,$$ we see that $q(K_X+B_i)+\Div(\alpha_i)$ is the pullback of a $\bQ$-Cartier $\bQ$-divisor $qL_Z$ on $Z$ by \cite[Lemma 2.5]{CHL23}. Thus we have the following adjunction formula
		\begin{equation*}
			K_X+B_i\sim_{q}f^*(K_Z+B_{i,Z}+M_{i,Z})
		\end{equation*}
		where $B_{i,Z}$ is the discriminant divisor and $M_{i,Z}=L_Z-K_Z-B_{i,Z}$ is the moduli divisor.
		
		 By Lemma \ref{adjunction formula for R coefficients} and Lemma \ref{moduli part Cartier}, there exists a positive integer $p$ depending only on $d,u,I,\Phi$ such that $pM_{i,\ZZ}$ is Cartier nef on some high resolution $\ZZ\to Z$.
		
		\item Now we have an adjunction formula
		\begin{equation*}
			(\ast) \quad	K_X+B\sim_{\bR}f^*(K_Z+B_Z+M_Z),
		\end{equation*}
		where $B_Z=\sum_{i=1}^l r_iB_{i,Z}$ and $\textbf{M}=\sum_{i=1}^l r_i\textbf{M}_i$. 	 Let $J^{\prime}=\{\frac{r}{p}|r\in J\}$, then $$M_{\ZZ}=\sum_{i=1}^l \frac{r_i}{p}(pM_{i,\ZZ})$$ where $\frac{r_i}{p}\in J^{\prime}$ and $pM_{i,\ZZ}$ is Cartier nef for $1\leq i \leq l$.
		
		\item In this step we prove the moreover part. Now we assume that $(X,B)$ is $\epsilon$-lc for some fixed $\epsilon\in \bR^{>0}$, then in Step 1, we can choose $(X,B_i)$ to be $\frac{\epsilon}{2}$-lc by \cite[Corollary 5.5]{HLS19}. By the argument of the proof of \cite[Lemma 6.7]{VGP} and Lemma \ref{singularities on base}, replacing $q$ with a bounded multiple, we can assume that $qB_{i,Z}$ is integral. Since $B_Z=\sum_{i=1}^l r_iB_{i,Z}$ and $r_i$'s are in a finite set, we conclude that $B_Z\in \Psi$ where $\Psi \subset \bR^{\geq 0}$ is a finite set.
	\end{enumerate}
\end{proof}

\subsection{DCC coefficients case}

Next we show that there exists an effective adjunction formula in the case that the coefficients of $B$ are in a fixed DCC set of real numbers.

\begin{thm}\label{effective adjunction formula}
	Let $d\in \bN$, $u\in \bR^{>0}$ and $\Phi\subset \bR^{\geq0}$ be a DCC set. Then there exists a finite set $I \subset \bR^{\geq0}$ depending only on  $d,u, \Phi$ satisfying the following. Assume that
	\begin{itemize}
		\item $(X,B)$ is  a projective klt pair of dimension $d$ and $B\in \Phi$,
		\item $f: X \to Z$ is a contraction with $K_X+B\sim_{\bR} 0/Z$,
		\item $A\in \Phi$ is an effective $\bR$-divisor on $X$ such that $A$ is relatively semi-ample over the generic point $\eta_Z$ of $Z$, and
		\item $0<\vol(A|_F)\leq u$ for the general fiber $F$ of $f$. 
	\end{itemize} 
	
	Then there is an adjunction formula 
	\begin{equation*}
		K_X+B \sim_{\bR} f^*(K_Z+B_Z+M_Z)
	\end{equation*}
	such that $M_{\ZZ}=\sum_{i=1}^l r_iM_{i,\ZZ}$ on some high resolution $\ZZ\to Z$, where $r_i\in I$ and  $M_{i,\ZZ}$ is Cartier nef for any $i$. 
\end{thm}

\begin{proof}
	\begin{enumerate} [label=\textsl{Step} \arabic{enumi}., wide=13pt, itemsep=13pt]
		
		\item Let $F$ be a general fiber of $f:X\to Z$. Then $K_F+B_F:=(K_X+B)|_F\sim_{\bR}0$ and $(F,B_F)$ is a klt log Calabi-Yau pair. By global ACC for lc thresholds \cite[Theorem 1.5]{ACCLCT}, since the coefficients of $B_F$ are in a DCC set $\Phi$, they are in a finite set $J\subset \bR^{\geq0}$ depending only on $d,\Phi$. Hence if we denote $B^h$ to be the horizontal/$Z$ part of $B$, then $B^h\in J$.		
		
		\item In this step and the next step, we obtain a new pair from which preserves the horizontal part of $B$ over the generic point of $Z$ but changes the vertical part of $B$ to be reduced.
		  
		  Take high log resolutions of $(X,B)$ and $Z$ as follows:
		\begin{displaymath}
			\xymatrix{
				\XX \ar[d]_{\ff} \ar[r]^\pi  & X \ar[d]^f         \\
				\ZZ \ar[r]^\mu  & Z   }
		\end{displaymath}
		such that $(\XX,\Sigma)$ is log smooth, where $\Sigma$ is the sum of reduced $\pi$-exceptional divisors and the birational transform of $\Supp B$. Write $K_{\XX}+\BB=\pi^*(K_X+B)$. Let $\tilde{B}^v,\tilde{B}^h$ be the vertical/$\ZZ$ part and horizontal/$\ZZ$ part of the birational transform of $B$. Let $E^v,E^h$ be the vertical/$\ZZ$ part and horizontal/$\ZZ$ part of the reduced $\pi$-exceptional divisors. Then  we take an open subset $\UU$ in $\ZZ$ such that 
		\begin{itemize}
			\item $\mu:\ZZ\to Z$ is an isomorphism on $\UU$,
			\item $L:=\ZZ\backslash \UU$ is a reduced divisor on $\ZZ$, and
			\item  $\ff(\Supp(\tilde{B^v}+E^v))\subseteq L$.
		\end{itemize}
		
		Let $\ff^{-1}L$ be the reduction of the inverse image of $L$ with respect to $\ff:\XX\to \ZZ$ and  add $\ff^{-1}L$ to $\Sigma$. Possibly replacing $(\XX,\Sigma)$ with a higher birational model, we can assume that the condition $(\XX,\Sigma)$ being log smooth is preserved.
		\item Let $\Gamma^{\prime}=\tilde{B^h}+E^h+\ff^{-1}L$. Replacing $J$ with $J\cup \{1\}$, we have $\Gamma^{\prime}\in J$. Run an MMP on $K_{\XX}+\Gamma^{\prime}$ over $\ZZ$ with scaling of some ample divisor. Since over $\ff^{-1}\UU$, $(X,B)$ is a weak lc model of $(\XX,\Gamma^{\prime})$, hence by \cite[Corollary 3.7]{Bir12}, $(\XX,\Gamma^{\prime})$ has a minimal model over $\UU$. Therefore, by \cite[Theorem 1.9]{Bir12}, the MMP terminates over $\ff^{-1}\UU$ and we reach a model $(W,\Gamma_W)$ such that  $K_W+\Gamma_W\sim_{\bR}0/\UU$.
		
		Now we continue to run the MMP on $K_W+\Gamma_W$ over $\ZZ$. The MMP does not modify $W$ over $\UU$. Moreover, the MMP is also an MMP on $K_W+\Gamma_W-aF_W$ where $F_W$ is the pullback of $L$ with respect to $W\to \ZZ$ and $a>0$ is a small number. Since $K_W+\Gamma_W-aF_W$ is semi-ample over $\UU$ and $(W,\Gamma_W-aF_W)$ is klt, the MMP terminates with a good minimal model $V$ by \cite[Theorem 1.2]{Hashizume19}. Let $g: V\to \ZZZ$ be the contraction induced by the semi-ample/$\ZZ$ $\bR$-divisor $K_V+\Gamma_V$ and denote by $\mu^{\prime}$ the morphism $\ZZZ\to \ZZ$. If we denote $K_V+B_V$ as the pushdown of $K_{\XX}+\BB$, then $\Supp(\Gamma_V-B_V)$ maps into $L\subseteq \ZZ$. Since $K_{\XX}+\BB\sim_{\bR}0/\ZZ$, by the cone theorem, the pullbacks of $K_V+B_V$ and $K_X+B$ to a common resolution are the same. Therefore, we conclude that $(V,B_V)$ is a sub-klt pair and $K_V+B_V\sim_{\bR}0/Z$. Let $A_V$ be the birational transform of the horizontal/$Z$ part of $A$. Let $G$ be the general fiber of $g:V\to \ZZZ$. Since over $\UU$, $(V,\Gamma_V)$ is a small $\bQ$-factorialization of $(X,B)$, $A_V$ is the pullback of $A$. Therefore, $A_V$ is relatively semi-ample over the generic point of $\ZZZ$ and $0<\vol(A_V|_G)\leq u$.  
		
		\item Applying Theorem \ref{finite effective adjunction} to $(V,\Gamma_V)$ over $\ZZZ$, there exists a finite set $I\subset \bR^{\geq0}$ depending only on $d,u,J, \Phi$ such that we can write an adjunction formula
		\begin{equation*}
			(\ast) \quad K_V+\Gamma_V\sim_{\bR} g^*(K_{\ZZZ}+\Gamma_{\ZZZ}+M_{\ZZZ})
		\end{equation*}
		such that $M_{\ZZZZ}=\sum_{i=1}^l r_iM_{i,\ZZZZ}$ on some high resolution $\ZZZZ\to \ZZZ$, where $r_i\in I$ and  $M_{i,\ZZZZ}$ is Cartier nef for any $i$.
		
		Since $K_V+\Gamma_V$ is the pullback of  $K_X+B$ over $\UU\subseteq \ZZ$, $\Gamma_V-B_V$ is vertical over $\ZZZ$. Since $$K_V+\Gamma_V\sim_{\bR}K_V+B_V\sim_{\bR}0/\ZZZ,$$ we conclude that $\Gamma_V-B_V$ is the pullback of an effective $\bR$-Cartier $\bR$-divisor $P_{\ZZZ}$ on $\ZZZ$ by \cite[Lemma 2.11]{effectiveadjunctionconjecture}. The adjunction formula $(\ast)$ induces an adjunction formula
		\begin{equation*}
			(\ast\ast) \quad K_V+B_V\sim_{\bR} g^*(K_{\ZZZ}+B_{\ZZZ}+M_{\ZZZ})
		\end{equation*}
		 where $B_{\ZZZ}:=\Gamma_{\ZZZ}-P_{\ZZZ}$. 
		
		Since the pullbacks of $K_V+B_V$ and $K_X+B$ to a common resolution are the same, the adjunction formula $(\ast\ast)$ induces the following adjunction formula 
		\begin{equation*}
			(\ast\ast\ast) \quad K_X+B\sim_{\bR} f^*(K_Z+B_Z+M_Z)
		\end{equation*}
		where $K_Z+B_Z+M_Z$ is the pushdown of $K_{\ZZZ}+B_{\ZZZ}+M_{\ZZZ}$. 
	\end{enumerate}
\end{proof}

\section{Singularities of klt stable minimal models}
In this section we apply the effective adjunction formula to prove Theorem \ref{singularities of klt stable minimal models}. We follow the proof of \cite[Lemma 8.2]{BVGP}. 

\begin{proof}[Proof of Theorem \ref{singularities of klt stable minimal models}]
	\begin{enumerate} [label=\textsl{Step} \arabic{enumi}., wide=13pt, itemsep=13pt]
		\item Let $(X,B),A\in \cS_{klt}(d,\Phi,\leq u,v)$ and $f:  X\to Z$ be the contraction defined by the semi-ample $\bR$-divisor $K_X+B$.  By Theorem \ref{effective adjunction formula}, there exists a finite set $I \subset \bR^{\geq0}$ depending only on  $d, u,\Phi$ such that there is an adjunction formula 
		\begin{equation*}
			(\ast) \quad K_X+B \sim_{\bR} f^*(K_Z+B_Z+M_Z)
		\end{equation*}
		satisfying that $M_{\ZZ}=\sum_{i=1}^l r_iM_{i,\ZZ}$ on some high resolution $\ZZ \to Z$, where $r_i\in I$ and  $M_{i,\ZZ}$ is Cartier nef for any $i$. By the ACC for lc thresholds \cite[Theorem 1.1]{ACCLCT}, the coefficients of $B_Z$ are in a DCC set $\Psi\subset \bR^{\geq0}$ depending only on $d, \Phi$.
		Since $$\vol(K_Z+B_Z+M_Z)=\Ivol(K_X+B)=v,$$
		we conclude that $(Z,B_Z+M_Z)\in \cF_{gklt}(\dim Z, \Psi\cup I, v)$. By \cite[Theorem 1.5]{BVGP}, $(Z,B_Z+M_Z)$ is generalized $\delta$-lc for some $\delta\in \bR^{>0}$ depending only on $\dim Z, \Psi,I, v$, hence depending only on $d, \Phi, u, v$.
		
		\item If $D$ is a prime divisor over $X$ which is horizontal over $Z$, then $D$ determines a prime divisor $S$ over the general fiber $F$ of $f$. Since $(F,B_F)$ is a klt log Calabi-Yau pair and $B_F\in \Phi$, by \cite[Lemma 2.48]{BABI}, $(F, B_F)$ is $\tau$-lc for some $\tau\in \bR^{>0}$ depending only on $\dim F$ and $\Phi$, hence depending only on $d, \Phi$.  Then 
		\begin{equation*}
			a(D, X, B)=a(S, F, B_F)\geq \tau.
		\end{equation*}
		
		\item If $D$ is a prime divisor over $X$ which is vertical over $Z$, then we take high resolutions as follows:
		
		\begin{displaymath}
			\xymatrix{
				\XX \ar[d]_-{f^\prime} \ar[r]^\pi  & X \ar[d]^f            \\
				\ZZ \ar[r] ^\mu & Z   }
		\end{displaymath}
		such that $D$ is a divisor on $\XX$ and its image on $\ZZ$ is a prime divisor $E$. Let 
		\begin{equation*}
			K_{\XX}+\BB=\pi^*(K_X+B) 
		\end{equation*}
		and
		\begin{equation*}
			K_{\ZZ}+B_{\ZZ}+M_{\ZZ}=\mu^*(K_Z+B_Z+M_Z).
		\end{equation*}
		Since $(Z,B_Z+M_Z)$ is generalized $\delta$-lc,
		\begin{equation*}
			a(E,\ZZ,B_{\ZZ}+M_{\ZZ})=a(E,Z,B_Z+M_Z)\geq \delta.
		\end{equation*}
		Therefore, $$\mult_E B_{\ZZ}\leq 1-\delta.$$ By the definition of discriminant divisors, $(\XX,\BB+\delta f^{\prime *}E)$ is sub-lc over the generic point of $E$. This implies that $$\mult_D \BB+\delta \mult_D f^{\prime *}E\leq 1$$ and hence $\mult_D \BB\leq 1-\delta$. Thus 
		\begin{equation*}
			a(D,X,B)=a(D,\XX,\BB)\geq \delta.
		\end{equation*}	
		\item  From the above arguments we see that $(X,B)$ is $\epsilon$-lc, where $\epsilon:=\min\{\tau, \delta\}$ depending only on $d, \Phi, u, v$.
	\end{enumerate}
	
\end{proof}
\section{Boundedness of klt stable minimal models}
\subsection{Lower bound on lc thresholds of strongly stable minimal models}

We prove there is  a uniform lower bound on lc thresholds of strongly stable minimal models with klt singularities, which is the $\bR$-coefficient version of \cite[Lemma 8.3]{BVGP} and the proof is similar to the one there. The new tools applied here are the effective adjunction formula with $\bR$-coefficients (Theorem \ref{effective adjunction formula}) and uniform decomposition of generalized $\bR$-pairs (Lemma \ref{decomposition of ample divisor}).

\begin{thm}\label{lct for strong minimal model}
	Let $d\in \bN$, $u,v,w\in \bR^{>0}$, and $\Phi\subset \bR^{\geq0}$ be a DCC set. Then there exists $\lambda\in \bR^{>0}$ depending only on $d,u,v,w,\Phi$ such that for any strongly stable minimal model $$(X,B),A\in \cSS_{klt}(d,\Phi,\leq u,v)$$ with $\vol(K_X+B+A)\leq w$, the pair $(X,B+\lambda A)$ is lc.
\end{thm}

\begin{proof}
	\begin{enumerate} [label=\textsl{Step} \arabic{enumi}.,wide=13pt,itemsep=13pt]
		\item Let $(X,B),A\in \cSS_{klt}(d,\Phi,\leq u,v)$.	By Theorem \ref{singularities of klt stable minimal models}, there is $\epsilon\in \bR^{>0}$ depending only on $d,u,v,\Phi$ such that $(X,B)$ is $\epsilon$-lc. Since $K_X+B\sim_{\bR} 0/Z$, we have $K_F+B_F=(K_X+B)|_F \sim_{\bR} 0$, where $F$ is the general fiber of $f$.  Since $A$ is ample over $Z$, $(F,B_F),A_F$ is a polarized $\epsilon$-lc log Calabi-Yau pair. Since the coefficients of $B_F,A_F$ are in a DCC set $\Phi$, they are bounded from below away from zero. 
		
		\item Let $s=\dim Z$. If $s=0$, then $(X,B),A$ is bounded by \cite[Theorem 6.2]{GOPV}, in this case the theorem follows from \cite[Theorem 1.8]{BABII}. Now we assume that $s\geq 1$. By Theorem \ref{effective adjunction formula}, there exists a finite set $I\subset \bR^{\geq0}$ depending only on $d,u,\Phi$ such that there is an adjunction formula
		\begin{equation*}
			K_X+B \sim_{\bR} f^*(K_Z+B_Z+M_Z)
		\end{equation*}
		satisfying that $M_{\ZZ}=\sum_{i=1}^l \mu_iM_{i,\ZZ}$ on some high resolution $\ZZ \to Z$, where $\mu_i\in I$ and  $M_{i,\ZZ}$ is Cartier nef for any $i$. By the ACC for lc thresholds \cite[Theorem 1.1]{ACCLCT}, the coefficients of $B_Z$ are in a DCC set $\Psi\subset \bR^{\geq0}$ depending only on $d, \Phi$.
		Hence $(Z,B_Z+M_Z)\in \cF_{gklt}(s, \Psi\cup I,v)$. By \cite[Theorem 1.5]{BVGP}, there is a fixed finite set $I^\prime\subset \bR^{>0}$ such that $B_Z\in I^\prime$. Therefore, $(Z,B_Z+M_Z)\in  \cF_{gklt}(s, I\cup I^\prime,v)$.
		
		By Theorem \ref{decomposition of ample divisor}, there is a fixed finite set $J\subset \bR^{\geq0}$ and a fixed $p\in \bN$ such that we can decompose $K_Z+B_Z+M_Z$ as follows:
		\begin{equation*}
			K_Z+B_Z+M_Z=\sum_{i=1}^l r_i(K_Z+B_{i,Z}+M_{i,Z})
		\end{equation*} 
		satisfying that 
		\begin{itemize}
			\item $\sum_{i=1}^l r_i=1$ and $r_i\in J$ for any $i$ ,
			\item $(Z,B_{i,Z}+M_{i,Z})$ is generalized klt for any $i$, and
			\item $p(K_Z+B_{i,Z}+M_{i,Z})$ is very ample for any $i$.
		\end{itemize}
		
		\item Let $$H:=\sum_{i=1}^lp(K_Z+B_{i,Z}+M_{i,Z})$$ and let $n$ be a fixed positive integer such that $n>\max\{\frac{p}{r_1},\cdots,\frac{p}{r_l}\}$, then $f^*H-(K_X+B)$ and $n(K_X+B)-f^*H$ are nef. Take $\pi: Y\to X$ as a $\bQ$-factorialization of $X$ and write $K_Y+B_Y=\pi^*(K_X+B)$. Since $(X,B),A$ is a strongly stable minimal model, $K_X+B+A$ is ample, thus $f^*H+A$ is ample, which implies that $\pi^*f^*H+\pi^*A$ is nef and big. Since $Y$ is $\epsilon$-lc and since $$\pi^*f^*H+\pi^*A-K_Y=\pi^*(f^*H-K_X-B)+\pi^*A+B_Y$$ is pseudo-effective and the coefficients of $\pi^*A$ are in the DCC set $\Phi$, by \cite[Theorem 4.2]{GOPV}, there is a bounded positive integer $m$ such that $|m(\pi^*f^*H+\pi^*A)|$ is birational, hence $|m(f^*H+A)|$ is birational.
		
		\item Replacing $m$ with a bounded multiple we can assume that $mA\geq 1$. Pick a member $$N\in|m(f^*H+A)|$$ and then 
		\begin{align*}
			\vol(N)&=\vol(m(f^*H+A))\\
			&\leq \vol(m(n(K_X+B)+A))\\
			&\leq (mn)^dw.
		\end{align*}
		For any component $D$ of $N$, we have $\mult_D(N+B+mA)\geq 1$ because any component of $N$ whose coefficient is not an integer is a component of $A$. By construction, $$N-(K_X+B+mA)\sim (m-1)f^*H+(f^*H-K_X-B)$$ is pseudo-effective. Applying \cite[Proposition 4.4]{BABI} to $(X,B+mA),N$, we conclude that there is $c\in \bR^{>0}$ and a bounded set $\cP$ depending only on $d, m,n,w$ such that there is a log smooth couple $(\overline{X},\overline{\Sigma})\in \cP$ and a birational map $\overline{X}\dasharrow X$ satisfying that 
		\begin{itemize}
			\item $\overline{\Sigma}$ contains the reduced exceptional divisor of $\overline{X}\dasharrow X$ and the support of the birational transform of $B+A+N$, and
			\item if $\overline{N}$ is the pullback of $N$ to $\overline{X}$, then $\overline{N}\cdot\overline{H}^{d-1}\leq c$ for some very ample divisor $\overline{H}\leq \overline{\Sigma}$.
		\end{itemize}
		The meaning of the pullback of $N$ to $\overline{X}$(and similarly for other divisors) is pulling back $N$ to a common resolution of $X,\overline{X}$ and then pushing down to $\overline{X}$. Note that in \cite[Proposition 4.4]{BABI}, $N$ is required to be a $\bQ$-divisor, but the proof of \cite[Proposition 4.4]{BABI} goes through verbatim when $N$ is an $\bR$-divisor. Let $\overline{A}$ be the pullback of $A$ to $\overline{X}$, then $\overline{A}\cdot\overline{H}^{d-1}\leq c$ since $N-A$ is pseudo-effective, and this implies that $\overline{A}\leq c$.
		
		\item Let $K_{\overline{X}}+\overline{B}$ be the pullback of $K_X+B$ to $\overline{X}$, then $\overline{B}\leq 1-\epsilon$ because $(X,B)$ is $\epsilon$-lc. Therefore, there is a fixed number $\lambda\in (0,1)$ such that $(\overline{X},\overline{B}+\lambda\overline{A})$ is sub-lc as $(\overline{X},\overline{B}+\overline{A})$ is log smooth.
		
		Since $$K_X+B+\lambda A=\lambda (K_X+B+A)+(1-\lambda)(K_X+B)$$ is ample, by negativity lemma we conclude that $(X,B+\lambda A)$ is lc.
	\end{enumerate}
\end{proof}

\subsection{Proof of Theorem \ref{mainthm1}}
We are now ready to prove the boundedness of stable minimal models with klt singularities. As explained in the sketch of the proofs, given a stable minimal model $(X,B),A$ in Theorem \ref{mainthm1}, the key point is to shrink $\Nlc(X,B+\lambda A)$ by induction. 
\begin{proof}[ Proof of Theorem \ref{mainthm1}]
	\begin{enumerate} [label=\textsl{Step} \arabic{enumi}.,wide=13pt,itemsep=13pt]
		\item Let $(X,B),A\in \cS_{klt}(d,\Phi,\leq u,v)$ such that $(K_X+B)^i\cdot A^{d-i}\leq w$ for $0\leq i\leq d$, and $f: X\to Z$ be the contraction induced by the semi-ample $\bR$-divisor $K_X+B$.  By Theorem \ref{singularities of klt stable minimal models}, there exists $\epsilon\in \bR^{>0}$ depending only on $d,\Phi,u,v$ such that $(X,B)$ is $\epsilon$-lc.
		
		Since $K_X+B\sim_{\bR} 0/Z$, we have $K_F+B_F=(K_X+B)|_F \sim_{\bR} 0$, where $F$ is the general fiber of $f$.  Thus $(F,B_F),A_F$ is a polarized $\epsilon$-lc log Calabi-Yau pair. Since the coefficients of $A, B$ are in a DCC set $\Phi$, they are bounded from below away from zero. Therefore, by \cite[Theorem 6.2]{GOPV}, $(F, \Supp(B_F+A_F))$ belongs to a bounded family. 
		
		\item Let $s=\dim Z$. If $s=0$, then by Step 1, $(X,B),A$ is bounded, hence we can assume that $s\geq 1$. By Theorem \ref{effective adjunction formula}, there exists a finite set $I\subset \bR^{\geq0}$ depending only on $d,u,\Phi$ such that there is an adjunction formula
		\begin{equation*}
			K_X+B \sim_{\bR} f^*(K_Z+B_Z+M_Z)
		\end{equation*}
		satisfying that $M_{\ZZ}=\sum_{i=1}^l \mu_iM_{i,\ZZ}$ on some high resolution $\ZZ\to Z$, where $\mu_i\in I$ and  $M_{i,\ZZ}$ is Cartier nef for any $i$. By the ACC for lc thresholds \cite[Theorem 1.1]{ACCLCT}, the coefficients of $B_Z$ are in a DCC set $\Psi\subset \bR^{\geq0}$. Hence $(Z,B_Z+M_Z)\in \cF_{gklt}(s, \Psi\cup I,v)$. By \cite[Theorem 1.5]{BVGP}, there is a finite set $I^\prime\subset \bR^{\geq 0}$ depending only on $s,v,\Psi,I$, hence depending only on $d,u,v,\Phi$, such that $B_Z\in I^\prime$. Therefore, $(Z,B_Z+M_Z)\in  \cF_{gklt}(s, I\cup I^\prime,v)$.
		
		By Theorem \ref{decomposition of ample divisor}, there is a fixed finite set $J\subset \bR^{\geq0}$ and a fixed $p\in \bN$ such that we can decompose $K_Z+B_Z+M_Z$ as follows:
		\begin{equation*}
			K_Z+B_Z+M_Z=\sum_{i=1}^l r_i(K_X+B_{i,Z}+M_{i,Z})
		\end{equation*} 
		satisfying that 
		\begin{itemize}
			\item $\sum_{i=1}^l r_i=1$ and $r_i\in J$ for any $i$,
			\item $(Z,B_{i,Z}+M_{i,Z})$ is generalized klt for any $i$, and
			\item $L_i:=p(K_Z+B_{i,Z}+M_{i,Z})$ is a very ample for any $i$.
		\end{itemize}		 
		
		\item Define $P(t)$ as the image of non-lc locus of $(X,B+tA)$ on $Z$, i.e. $$P(t):=f(\Nlc(X,B+tA)).$$ We regard $P(t)$ as a closed subset of $Z$. Now we prove the following statement $\cC_k$ by induction on $k$: there exists $\lambda_k\in \bR^{>0}$ depending only on $d,\Phi,u,v,w$ such that 
		\begin{itemize}
			\item $\dim P(\lambda_k)\leq s-k$, and
			\item $K_X+B+\lambda_k A$ is nef in dimension $k-1$ over $Z$.
		\end{itemize}
		
		By \cite[Theorem 1.8]{BABII}, there exists $\lambda_1\in \bR^{>0}$ such that $(F,B_F+\lambda_1 A_F)$ is lc. Then $(X,B+\lambda_1 A)$ is lc over some open subset $U\subseteq Z$, hence $\dim P(\lambda_1)\leq s-1$.  Since $A$ is relatively nef over $Z$ and $K_X+B$ is semi-ample, we conclude that $K_X+B+\lambda_1 A$ is nef in dimension $0$ over $Z$.
		
		Assume that the statement $\cC_k$ holds and $\lambda_k$ is chosen for $\cC_k$. We prove the statement $\cC_{k+1}$ in the following steps.
		
		\item In this step we prove that there exists a fixed $\lambda_{k+1}\in \bR^{>0}$ such that $K_X+B+\lambda_{k+1}A$ is nef in dimension $k$ over $Z$.
		
		Let $H_1,\cdots,H_{s-k}$ be any very ample divisors in $Z$. Take  a general member of the linear system $|H_i|$ and in abuse of notation, we still denote by $H_i$ the general member. Let $T=\cap_{i=1}^{s-k} H_i$, $R_i=f^{-1}H_i$ and $S=\cap_{i=1}^{s-k}R_i=f^{-1}T$. We can write 
		$$K_S+B_S+\lambda_k A_S=(K_X+B+\lambda_k A+\sum_{i=1}^{s-k}R_i)|_S,$$ where $A_S=A|_S$.
		
		If $K_S+B_S+\lambda_k A_S$ is nef, then take $\lambda_{k+1}=\lambda_k$ and we are done. Now assume that $K_S+B_S+\lambda_k A_S$ is not nef. By \cite[Theorem 1.2, Corollary 1.4]{fujinoadjunction}, $$\Nlc(S,B_S+\lambda_k A_S)=S\cap \Nlc(X,B+\lambda_k A).$$ Since $\dim P(\lambda_k)\leq s-k$, $S=\cap_{i=1}^{s-k}f^{-1}H_i$, and $H_i$ are general hypersurfaces in $Z$, we deduce that $\Nlc(S,B_S+\lambda_k A_S)$ is contained in finitely many fibers of $S\to T$. Let $R$ be a $(K_S+B_S+\lambda_k A_S)$-negative extremal ray. Since $K_S+B_S+\lambda_kA_S$ is ample over $T$ and $K_S+B_S$ is the pullback of an ample divisor on $T$, $K_S+B_S$ is positive on $R\backslash\{0\}$, which implies that $R$ is not contained in the image of the map of the closed cone of curves $$\overline{NE}(\Nlc(S,B_S+\lambda_k A_S))\to\overline{NE}(S).$$ By the length of extremal ray \cite[Theorem 1.1(5)]{fundamentaltheoremforlmmp}, there is a curve $C$ on $S$ generating $R$ and satisfying that $$(K_S+B_S+\lambda_k A_S)\cdot C\geq -2(d-s+k).$$
		
		Take a fixed positive integer $$m>\max\{\frac{p}{r_1},\cdots,\frac{p}{r_l}\},$$ then $m(K_Z+B_Z+M_Z)-L_i$ is ample and thus $m(K_X+B)-f^*L_i$ is nef. Therefore, $$m(K_X+B)|_S\cdot C\geq f^*L_1|_S\cdot C=L_1\cdot f(C)\geq1$$ because $L_1$ is a very ample Cartier divisor on $Z$. Now we have
		\begin{align*}
			&\left((1+2(d-s+k)m)(K_X+B+\sum_{i=1}^{s-k}R_i)|_S+\lambda_k A_S\right)\cdot C\\
			\geq& (K_S+B_S+\lambda_k A_S)\cdot C+ 2(d-s+k)m(K_X+B)|_S \cdot C\geq0.
		\end{align*}
		
		Let $$\lambda_{k+1}=\frac{\lambda_k}{1+2(d-s+k)m}.$$ The above argument  implies that $K_S+B_S+\lambda_{k+1}A_S$ is nef. Therefore, $K_X+B+\lambda_{k+1}A$ is nef in dimension $k$ over $Z$.

		\item Let $H_1,\cdots,H_{s-k}$ be general members of one of the linear systems $|L_i|$ (each $H_j$ can be in different linear systems). Recall that $m(K_X+B)-f^*L_i$ is nef for $1\leq i\leq l$, thus $m(s-k)(K_X+B)-\sum_{i=1}^{s-k}f^*H_i$ is nef. By Lemma \ref{nef in dimension}, we have the following inequality:
		\begin{equation*}
			(\ast) \quad ((1+m(s-k))(K_X+B)+\lambda_{k+1}A)^{d-s+k}\cdot f^*H_1\cdot \ldots \cdot f^*H_{s-k}\geq0.
		\end{equation*}
		
		\item In this step we prove that after replacing $\lambda_{k+1}$ with a smaller number, $\dim P(\lambda_{k+1})\leq s-k-1$ and hence $\cC_{k+1}$ is satisfied.

		Let $Q_1,\cdots, Q_{s-k}$ be general members of the linear system $|L_1|$. Abusing notation denote $T=\cap_{i=1}^{s-k}Q_i$, $R_i=f^*Q_i$ and $S=\cap_{i=1}^{s-k}R_i=f^{-1}T$. We can write 
		$$K_S+B_S+\lambda_{k+1} A_S=(K_X+B+\lambda_{k+1} A+\sum_{i=1}^{s-k}R_i)|_S,$$ where $A_S=A|_S$. 
		\begin{itemize}[itemsep=13pt]
			
			\item Since $K_X+B+\lambda_{k+1}A$ is nef in dimension $k$ over $Z$, $K_S+B_S+\lambda_{k+1}A_S$ is nef. We claim that $K_S+B_S+tA_S$ is ample for any $t\in(0,\lambda_{k+1})$: Fix $t\in(0,\lambda_{k+1})$. Since $K_S+B_S$ is the pullback of an ample divisor on $T$ and $A_S$ is ample over $T$, there is a sufficiently small $0<t^\prime<t$ such that $K_S+B_S+t^\prime A_S$ is ample. Then $K_S+B_S+tA_S$ is a positive linear combination of $K_S+B_S+t^\prime A_S$ and $K_S+B_S+\lambda_{k+1}A_S$, which implies that $K_S+B_S+tA_S$ is ample.
			
			Therefore, replacing $\lambda_{k+1}$ with $\frac{\lambda_{k+1}}{2}$ we can assume that $K_S+B_S+\lambda_{k+1}A_S$ is ample. Note that the condition $K_X+B+\lambda_{k+1}A$ being nef in dimension $k$ over $Z$ is preserved.
			
			\item Let $\overline{F}$ be the general fiber of $S\to T$ and $F$ be the general fiber of $X\to Z$. Since $T$ is the complete intersection of general hypersurfaces, $$\vol(A_S|_{\overline{F}})=\vol(A|_F)\leq u.$$
			\item Note that $$(\sum_{i=1}^{l}r_iL_i)^s=\vol(p(K_Z+B_Z+M_Z))=p^sv.$$ If $(\alpha_1,\cdots,\alpha_l)$ is an element of the set $$\{(n_1,\cdots,n_l)|n_i\in \bN,\sum_{i=1}^ln_i=s\},$$ then the intersection number $$L_1^{\alpha_1}\cdot \ldots \cdot L_l^{\alpha_l}$$ is a positive integer and bounded from above. Without loss of generality, we can assume that $$L_1^{\alpha_1}\cdot \ldots \cdot L_l^{\alpha_l}$$ are fixed numbers for all $$(\alpha_1,\cdots,\alpha_l)\in\{(n_1,\cdots,n_l)|n_i\in\bN,\sum_{i=1}^ln_i=s\}.$$
			Therefore, 
			\begin{align*}
				&\Ivol(K_S+B_S)\\
				=&\Ivol((K_X+B+\sum_{i=1}^{s-k}R_i)|_S)\\
				=&\vol((K_Z+B_Z+M_Z+\sum_{i=1}^{s-k}Q_i)|_T)\\
				=&(\frac{1}{p}\sum_{i=1}^l r_iL_i+(s-k)L_1)^k\cdot L_1^{s-k}
			\end{align*}
			is a fixed positive real number, say $v_k$.
			
			\item Now we show that $\vol(K_S+B_S+\lambda_{k+1}A_S)$ is bounded from above.
			\begin{align*}
				&\vol(K_S+B_S+\lambda_{k+1}A_S)\\
				=&\vol((K_X+B+\lambda_{k+1}A+\sum_{i=1}^{s-k}R_i)|_S)\\
				\leq&\vol(((1+m(s-k))(K_X+B)+\lambda_{k+1}A)|_S)\\
				=&((1+m(s-k))(K_X+B)+\lambda_{k+1}A)^{d-s+k}\cdot (f^*L_1)^{s-k}\\
				\leq & ((1+m(s-k))(K_X+B)+\lambda_{k+1}A)^{d-s+k}\cdot (m(K_X+B))^{s-k}
			\end{align*}
			is bounded from above because $(K_X+B)^i\cdot A^{d-i}\leq w$ for $0\leq i\leq d$. The first inequality follows from the nefness of  $m(s-k)(K_X+B)-\sum_{i=1}^{s-k}R_i$. The second inequality follows from the inequality $(\ast)$ in Step 5. 
		\end{itemize}
		~\\
		The above argument implies that $$(S,B_S),\lambda_{k+1}A_S\in \cSS_{klt}(d-s+k,\Phi\cup \lambda_{k+1}\Phi,\leq \lambda_{k+1}^{d-s}u, v_k).$$ Hence we can apply Theorem \ref{lct for strong minimal model} to $(S,B_S),\lambda_{k+1}A_S$. Replacing $\lambda_{k+1}$ with a smaller fixed number, we can assume that $(S,B_S+\lambda_{k+1}A_S)$ is lc. By inversion of adjunction \cite[Theorem 0.1]{haconinversion} or \cite{inversionofadjunction}, we conclude that $(X,B+\lambda_{k+1}A)$ is lc near $S$. Since $S=f^{-1}T$ and $T$ is the complete intersection of $(s-k)$ general hypersurfaces, we deduce that $\dim P(\lambda_{k+1})\leq s-k-1$, hence $\cC_{k+1}$ is satisfied.
		
		\item In this step we prove that there exists a fixed $\lambda\in \bR^{>0}$ such that $(X,B+\lambda A)$ is lc and $K_X+B+\lambda A$ is globally nef.
		
		By induction on $k$, $\cC_{s}$ is satisfied for some fixed $\lambda_{s}\in \bR^{>0}$.  Then $(X,B+\lambda_{s}A)$ is lc over the complement of  a finite set of closed points of $Z$ because $\dim P(\lambda_{s})\leq 0$.
		
		If $K_X+B+\lambda_{s}A$ is nef, then let $\lambda=\lambda_{s}$. If $K_X+B+\lambda_{s} A$ is not nef, let $R$ be a $(K_X+B+\lambda_{s} A)$-negative extremal ray. Repeating the process in Step 4, there is a curve $C$ generating $R$ and satisfying that $$(K_X+B+\lambda_{s} A)\cdot C\geq -2d$$ and $$m(K_X+B)\cdot C\geq 1.$$ Hence $$((1+2md)(K_X+B)+\lambda_{s} A)\cdot R\geq 0.$$ This argument implies that $$(1+2md)(K_X+B)+\lambda_{s} A$$ is nef. Let $\lambda=\frac{\lambda_{s}}{1+2md}$, then $K_X+B+\lambda A$ is nef. 
		
		The same argument as in Step 6 implies that after replacing $\lambda$ with $\frac{\lambda}{2}$, $K_X+B+\lambda A$ is ample. Therefore,
		\begin{equation*}
			(X,B), \lambda A\in \cSS_{klt}(d,\Phi\cup \lambda\Phi,\leq \lambda^{d-s}u,v).
		\end{equation*} 
		Moreover,
		\begin{equation*}
			\vol(K_X+B+\lambda A)=(K_X+B+\lambda A)^d
		\end{equation*}
		is bounded from above because $(K_X+B)^i\cdot A^{d-i}\leq w$ for $0\leq i\leq d$.
		By Theorem \ref{lct for strong minimal model}, after replacing $\lambda$ with a smaller fixed number, we can assume that $(X,B+\lambda A)$ is lc. 
		
		\item In this step we finish the proof. Since $(X,B)$ is $\epsilon$-lc, $(X,B+\frac{\lambda}{2}A)$ is $\frac{\epsilon}{2}$-lc. Since the coefficients of $B+\frac{\lambda}{2}A$ is in the DCC set $\Phi\cup \frac{\lambda}{2}\Phi$, there is a fixed $N\in \bN$ such that $|N(K_X+B+\frac{\lambda}{2}A)|$ is birational by \cite[Theorem 1.3]{ACCLCT}. Thus $|K_X+(2d+1)N(K_X+B+\frac{\lambda}{2}A)|$ is birational by \cite[Lemma 2.3.4]{birationalatuomorphism}. Since $\vol(K_X+B+\frac{\lambda}{2}A)$ is bounded from above, we conclude that $(X,B+\frac{\lambda}{2}A)$ is log birationally bounded by \cite[Theorem 3.1]{birationalatuomorphism}. Note that $K_X+B+\frac{\lambda}{2}A$ is ample, thus $(X,B+\frac{\lambda}{2}A)$ and hence $(X,B),A$ belongs to a bounded family by \cite[Theorem 1.6]{ACCLCT}.
	\end{enumerate}
\end{proof}

\begin{remark}\label{remove condition}
	The condition ``$\vol(A|_F)\leq u$'' in Theorem \ref{mainthm1} can be removed: applying \cite[Lemma 4.12]{moduliofalgebraicvarieties}, we have 
	\begin{equation*}
		(K_X+B)^{\dim Z}\cdot A^{d-\dim Z}=\Ivol(K_X+B)\vol(A|_F),
	\end{equation*}
	thus $\vol(A|_F)\leq \frac{w}{v}$.
\end{remark}

\subsection{Proof of Theorem \ref{mainthm2}} In this subsection we first generalize \cite[Theorem 1.3]{filipazzi2020some} to the generalized pair case.  Then we apply Theorem \ref{mainthm1} to prove Theorem \ref{mainthm2}.

\begin{lem}\label{volume fixed}
	Let $d\in \bN$, $\epsilon,v\in \bR^{>0}$, and $I\subset \bR^{\geq0}$ be a finite set. Then there exists a finite set $J\subset \bR^{\geq0}$ depending only on $d,\epsilon,v,I$ satisfying the following. 
	
	If $(X,B+M)\in \cF_{gklt}(d,I,\leq v)$ is a generalized $\epsilon$-lc pair, then $\vol(K_X+B+M)\in J$.
\end{lem}

\begin{proof}
	We follow the proof of \cite[Theorem 1.3]{filipazzi2020some}. 
	\begin{enumerate} [label=\textsl{Step} \arabic{enumi}.,wide=13pt,itemsep=13pt]
		\item Pick a generalized $\epsilon$-lc pair $(X,B+M)\in \cF_{gklt}(d,I,\leq v)$ with data $\XX\to X$ and $\MM$. Write $\MM=\sum_{i=1}^{l}\mu_i \MMi$, where $\mu_i\in I$ and $\MMi$ is Cartier nef for any $i$. By \cite[Theorem 1.2]{BVGP} and its proof, there exists a bounded set of couples $\cP$ depending only on $d,I,v$ such that there exists a log smooth couple $(\oX,\oSigma)\in \cP$ and a birational map $\pi:\oX\dasharrow X$ satisfying the following:
		\begin{itemize}
			\item $\oSigma$ contains the exceptional divisors of $\pi$ and the support of the birational transform of $B$,
			\item each $\MMi$ descends to $\oX$ as $\oM_i$, and
			\item $\oA-\sum_{i=1}^l \mu_i\oM_i$ is pseudo-effective for some very ample divisor $\oA\leq \oSigma$.
		\end{itemize}
		In particular,  $\MM$ descends to $\oX$ as $\oM=\sum_{i=1}^l \mu_i\oM_i$ and the number of $M_i$ is bounded from above.
		
		Let $\oB=(1-\epsilon)E+\tilde{B}$, where $E$ is the sum of exceptional divisors of $\pi$ and $\tilde{B}$ is the strict transform of $B$. Since $K_X+B+M$ is ample, by negativity lemma, if $D$ is a prime divisor on $X$ which is exceptional over $\oX$, then $$1\geq a(D,X,B+M)\geq a(D,\oX,\oB+\oM)\geq a(D,\oX,\oB).$$ Since $(\oX,\oB)$ is $\epsilon$-lc and belongs to a bounded family, we can extract all such $D$ and obtain a birational model $(\oX^{\prime},\oSigma^{\prime})$ which is also in a bounded family by \cite[Theorem 1.2]{FTF}, where $\oSigma^{\prime}$ is the sum of all exceptional divisors of $\oX^{\prime}\to \oX$ and the strict transform of $\oSigma$.  Replacing $(\oX,\oSigma)$ with a log bounded resolution of $(\oX^{\prime},\oSigma^{\prime})$, we can assume that $\pi^{-1}: X\dasharrow \oX$ does not contract divisors. We still denote $\oB=(1-\epsilon)E+\tilde{B}$, where $E$ is the sum of exceptional divisors of $\pi$ and $\tilde{B}$ is the strict transform of $B$.
		
		\item Applying Noetherian induction, we can assume that there is a log smooth couple $(\oV,\oGamma)$ and  a smooth projective morphism $\oV\to T$ onto a smooth variety, such that there is a closed point $t\in T$ so that we can identify $\oX$ with $\oV_t$ and $\oSigma\leq \oGamma_t$. Possibly taking a finite base change and shrinking $T$ we can assume that $(\oV, \oGamma)$ is log smooth over $T$. 
		
		Since the coefficients of $B$ are in a finite set $I$, without loss of generality, we can assume that the number of components of $B$ and their coefficients are fixed. Hence there exists a boundary $\oDelta$ on $\oV$ such that we can identify $\oDelta_t$ with $\oB$.
		
		By the argument of Step 3 and Step 4 in the proof of \cite[Theorem 1.3]{BVGP}, there exist divisors $\oN_i$ on $\oV$ such that $\oN_i|_{\oV_t}\sim_{\bQ} \oM_i$. If we write $\oN=\sum_{i=1}^l \mu_i \oN_i$, then $\oN|_{V_t}\sim_{\bR}\oM$. Since $\mu_i$'s belong to a finite set $I$, without loss of generality, we can assume that $\mu_i$'s are all fixed. Therefore, we regard $\oN$ as a fixed divisor on $\oV$.
		
		\item Applying \cite[Theorem 1.12]{filipazzi2018boundedness} to $(\oV,\oDelta),\oN$, we deduce that $\vol(K_{\oX}+\oB+\oM)$ is fixed. Since $(X,B+M)$ is generalized $\epsilon$-lc, $\oB=(1-\epsilon)E+\tilde{B}$ and $X\dasharrow \oX$ does not contract divisors, we have $$\vol(K_{\oX}+\oB+\oM)=\vol(K_X+B+M).$$ Therefore, $\vol(K_X+B+M)$ is fixed and we finish the proof.
	\end{enumerate}
\end{proof}

\begin{proof}[Proof of Theorem \ref{mainthm2}]
Let $(X,B),A\in \cS_{klt}(d,\Phi,\leq u,\leq v)$ such that $(X,B)$ is $\epsilon$-lc, and let $f: X\to Z$ be the contraction induced by the semi-ample $\bR$-divisor $K_X+B$. By Theorem \ref{finite effective adjunction}, there exist finite sets $J, \Psi\subset \bR^{\geq0}$ depending only on $d,u,\epsilon,\Phi$, such that there is an adjunction formula $$K_X+B\sim_{\bR}f^*(K_Z+B_Z+M_Z)$$ satisfying the following:
	\begin{itemize}
		\item $B_Z\in \Psi$, and
		\item $M_{\ZZ}=\sum_{i=1}^l r_iM_{i,\ZZ}$ on some high resolution $\ZZ\to Z$, where $r_i\in J$ and  $M_{i,\ZZ}$ is Cartier nef for any $i$. 
	\end{itemize}
	
	Therefore, we have $$(Z,B_Z+M_Z)\in \cF_{gklt}(d,J\cup \Psi,\leq v).$$ By Lemma \ref{singularities on base}, $(Z,B_Z+M_Z)$ is generalized $\delta$-lc for some $\delta$ depending only on $d,\epsilon$. Then Lemma \ref{volume fixed} implies that $$\Ivol(K_X+B)=\vol(K_Z+B_Z+M_Z)$$ is in a finite set depending only on $d,\delta,v,J,\Psi$, hence depending only on $d,\epsilon,u,v,\Phi$. Without loss of generality, we can assume that $\Ivol(K_X+B)=v_0$ for some fixed $v_0\in \bR^{>0}$. Therefore, $(X,B),A\in \cS_{klt}(d,\Phi,\leq u, v_0)$, and the theorem follows from Theorem \ref{mainthm1}.
\end{proof}

\hfill

\providecommand{\bysame}{\leavevmode\hbox to3em{\hrulefill}\thinspace}
\providecommand{\MR}{\relax\ifhmode\unskip\space\fi MR }
\providecommand{\MRhref}[2]{%
	\href{http://www.ams.org/mathscinet-getitem?mr=#1}{#2}
}
\providecommand{\href}[2]{#2}


\begin{thebibliography}{HMX18}
	
	\bibitem[Amb99]{ambroadjunctionformula}
	F.~Ambro, \emph{{The Adjunction Conjecture and its applications}},
	arXiv:math/9903060v3 (1999).
	
	\bibitem[Amb05]{ambromoduli}
	Florin Ambro, \emph{The moduli {$b$}-divisor of an lc-trivial fibration},
	Compos. Math. \textbf{141} (2005), no.~2, 385--403.
	
	\bibitem[BH22]{VGP}
	C.~Birkar and C.~D. Hacon, \emph{Variations of generalised pairs},
	arXiv:2204.10456 (2022).
	
	\bibitem[Bir12]{Bir12}
	C.~Birkar, \emph{Existence of log canonical flips and a special {LMMP}}, Pub.
	Math. IHES. \textbf{115} (2012), 325--368.
	
	\bibitem[Bir19]{BABI}
	C.~Birkar, \emph{{Anti-pluricanonical systems on Fano varieties}}, Ann. of Math.
	(2) \textbf{190} (2019), no.~2, 345 -- 463.
	
	\bibitem[Bir21a]{BVGP}
	C.~Birkar, \emph{Boundedness and volume of generalised pairs}, arXiv:2103.14935
	(2021).
	
	\bibitem[Bir21b]{BABII}
	C.~Birkar, \emph{{Singularities of linear systems and boundedness of Fano
			varieties}}, Ann. of Math. (2) \textbf{193} (2021), no.~2, 347--405.
	
	\bibitem[Bir22]{moduliofalgebraicvarieties}
	C.~Birkar, \emph{Moduli of algebraic varieties}, arXiv:2211.11237 (2022).
	
	\bibitem[Bir23a]{GOPV}
	C.~Birkar, \emph{Geometry of polarised varieties}, Pub. Math. IHES. \textbf{137}
	(2023), 47–105.
	
	\bibitem[Bir23b]{singularitiesonFanofibrations}
	C.~Birkar, \emph{{Singularities on Fano fibrations and beyond}}, arXiv:2305.18770
	(2023).
	
	\bibitem[Bir24]{FTF}
	C.~Birkar, \emph{{Boundedness of Fano type fibrations}}, Ann. Sci. \'Ec. Norm.
	Sup\'er. (4) \textbf{57} (2024), no.~3, 787--840.
	
	\bibitem[BZ16]{effectiveIitaka}
	C.~Birkar and D-Q. Zhang, \emph{{Effectivity of Iitaka fibrations and
			pluricanonical systems of polarized pairs}}, Pub. Math. IHES. \textbf{123}
	(2016), no.~1, 283--331.
	
	\bibitem[Che23]{CGD20}
	G.~Chen, \emph{Boundedness of $n$-complements for generalized pairs}, Eur. J.
	Math. \textbf{9} (2023), no.~4, 95.
	
	\bibitem[CHL23]{CHL23}
	G.~Chen, J.~Han, and J.~Liu, \emph{On effective log {Iitaka} fibrations and
		existence of complements}, Int. Math. Res, Not. \textbf{2024} (2023), no.~10,
	8329--8349.
	
	\bibitem[Cho08]{invariantIitakadimension}
	S.~R. Choi, \emph{The geography of log models and its applications}, Ph.D.
	thesis, Johns Hopkins University, 2008.
	
	\bibitem[FH23]{fujinoadjunction}
	O.~Fujino and K.~Hashizume, \emph{Adjunction and inversion of adjunction},
	Nagoya Math. J. \textbf{249} (2023), 119--147.
	
	\bibitem[FHS24]{FHS2021boundedness}
	S.~Filipazzi, C.~D. Hacon, and R.~Svaldi, \emph{{Boundedness of elliptic
			Calabi-Yau threefolds}}, J. Eur. Math. Soc. (2024), published online first.
	
	\bibitem[Fil18]{filipazzi2018boundedness}
	S.~Filipazzi, \emph{Boundedness of log canonical surface generalized polarized
		pairs}, Taiwanese J. Math. \textbf{22} (2018), no.~4, 813--850.
	
	\bibitem[Fil20]{filipazzi2020some}
	S.~Filipazzi, \emph{Some remarks on the volume of log varieties}, Proc. Edinb. Math.
	Soc. (2) \textbf{63} (2020), no.~2, 314--322.
	
	\bibitem[Fil24]{filipazzi2024boundedness}
	S.~Filipazzi, \emph{On the boundedness of n-folds with $\kappa(x)=n-1$}, Algebr.
	Geom. \textbf{11} (2024), no.~3, 318--345.
	
	\bibitem[FS20]{filipazziSvaldi20}
	S.~Filipazzi and R.~Svaldi, \emph{Invariance of plurigenera and boundedness for
		generalized pairs}, Mat. Contemp. \textbf{47} (2020), 114--150.
	
	\bibitem[Fuj11]{fundamentaltheoremforlmmp}
	O.~Fujino, \emph{Fundamental theorems for the log minimal model program}, Publ.
	Res. Inst. Math. Sci. \textbf{47} (2011), no.~3, 727--789.
	
	\bibitem[Fuj17]{fujinoeffective}
	O.~Fujino, \emph{Effective basepoint-free theorem for semi-log canonical
		surfaces}, Publ. Res. Inst. Math. Sci. \textbf{53} (2017), no.~3, 349--370.
	
	\bibitem[Hac14]{haconinversion}
	C.~D. Hacon, \emph{On the log canonical inversion of adjunction}, Proc. Edinb.
	Math. Soc. (2) \textbf{57} (2014), no.~1, 139--143.
	
	\bibitem[Has19]{Hashizume19}
	K.~Hashizume, \emph{Remarks on special kinds of the relative log minimal model
		program}, Manuscripta Math. \textbf{160} (2019), no.~3-4, 285--314.
	
	\bibitem[HH23]{HH23boundedness}
	K.~Hashizume and M.~Hattori, \emph{{On boundedness and moduli spaces of
			K-stable Calabi-Yau fibrations over curves}}, arXiv:2305.01244, to appear in
	Geom. Topol. (2023).
	
	\bibitem[HL22]{weakzariskidecomposition}
	J.~Han and Z.~Li, \emph{Weak {Zariski} decompositions and log terminal models
		for generalized pairs}, Math. Z. \textbf{302} (2022), no.~2, 707--741.
	
	\bibitem[HLQ23]{HLQ23}
	J.~Han, Y.~Liu, and L.~Qi, \emph{{ACC} for local volumes and boundedness of
		singularities}, J. Algebraic Geom. \textbf{32} (2023), no.~3, 519--583.
	
	\bibitem[HLS24]{HLS19}
	J.~Han, J.~Liu, and V.~V. Shokurov, \emph{{ACC for minimal log discrepancies of
			exceptional singularities}}, Peking Math. J. (2024), published online first.
	
	\bibitem[HLX23]{effectiveadjunctionformulawithrealcoefficients}
	J.~Han, J.~Liu, and Q.~Xue, \emph{{On the equivalence between the effective
			adjunction conjectures of Prokhorov-Shokurov and of Li}}, arXiv:2312.15397,
	to appear in Algebra Number Theory (2023).
	
	\bibitem[HMX13]{birationalatuomorphism}
	C.~D. Hacon, J.~M\textsuperscript{c}Kernan, and C.~Xu, \emph{On the birational
		automorphisms of varieties of general type}, Ann. of Math. (2) \textbf{177} (2013), no.~3,
	1077--1111.
	
	\bibitem[HMX14]{ACCLCT}
	C.~D. Hacon, J.~M\textsuperscript{c}Kernan, and C.~Xu, \emph{{ACC for log
			canonical thresholds}}, Ann. of Math. (2) \textbf{180} (2014), no.~2,
	523--571.
	
	\bibitem[HMX18]{HMX18}
	C.~D. Hacon, J.~M\textsuperscript{c}Kernan, and C.~Xu, \emph{Boundedness of
		moduli of varieties of general type}, J. Eur. Math. Soc. \textbf{20} (2018),
	no.~4, 865--901.
	
	\bibitem[HX13]{HX13}
	C.~D. Hacon and C.~Xu, \emph{Existence of log canonical closures}, Invent.
	Math. \textbf{192} (2013), no.~1, 161--195.
	
	\bibitem[HX15]{HX15}
	C.~D. Hacon and C.~Xu, \emph{{Boundedness of log Calabi-Yau pairs of Fano type}}, Math. Res.
	Lett. \textbf{22} (2015), no.~6, 1699--1716.
	
	\bibitem[Jia22]{jiao2022boundedness}
	J.~Jiao, \emph{{Boundedness of polarized Calabi-Yau fibrations}},
	arXiv:2202.07238, to appear in J. Differential Geom. (2022).
	
	\bibitem[Jia23]{jiang23boundedness}
	X.~Jiang, \emph{Boundedness and moduli of traditional stable minimal models},
	arXiv:2312.03313 (2023).
	
	\bibitem[JLX22]{JLX22}
	J.~Jiao, J.~Liu, and L.~Xie, \emph{On generalized lc pairs with b-log abundant
		nef part}, arXiv:2202.11256 (2022).
	
	\bibitem[Kaw98]{kawamataadjunctionformula}
	Yujiro Kawamata, \emph{Subadjunction of log canonical divisors. {II}}, Amer. J.
	Math. \textbf{120} (1998), no.~5, 893--899.
	
	\bibitem[Kaw07]{inversionofadjunction}
	M.~Kawakita, \emph{Inversion of adjunction on log canonicity}, Invent. Math.
	\textbf{167} (2007), no.~1, 129--133.
	
	\bibitem[Kol93]{effectivebasepointfree}
	J.~Koll{\'a}r, \emph{Effective base point freeness}, Math. Ann. \textbf{296}
	(1993), no.~4, 595--605.
	
	\bibitem[Kol23]{familiesofgeneraltype}
	J.~Koll{\'a}r, \emph{Families of varieties of general type}, Cambridge University
	Press, 2023.
	
	\bibitem[Li24a]{Iitakavolume}
	Z.~Li, \emph{Boundedness of the base varieties of certain fibrations}, J. Lond.
	Math. Soc. (2) \textbf{109} (2024), no.~2, e12871.
	
	\bibitem[Li24b]{effectiveadjunctionconjecture}
	Z.~Li, \emph{A variant of the effective adjunction conjecture with
		applications}, J. Pure Appl. Algebra \textbf{228} (2024), no.~6, 107626, 22.
	
	\bibitem[Xie22]{Contractiontheoremforgeneralizedpairs}
	L.~Xie, \emph{Contraction theorem for generalized pairs}, arXiv:2211.10800
	(2022).
	
	\bibitem[XZ24]{boundednessofsingularitiesIII}
	C.~Xu and Z.~Zhuang, \emph{Boundedness of log fano cone singularities and
		discreteness of local volumes}, arXiv:2404.17134 (2024).
	
	\bibitem[Zhu23]{boundednessofsingularitiesII}
	Z.~Zhuang, \emph{{On boundedness of singularities and minimal log discrepancies
			of {K}oll\'ar components, II}}, To appear in Geom. Topol., arXiv:2302.03841
	(2023).
	
	\bibitem[Zhu24]{boundednessofsingularitiesI}
	Z.~Zhuang, \emph{On boundedness of singularities and minimal log discrepancies of
		{K}oll\'ar components}, J. Algebraic Geom. \textbf{33} (2024), no.~3,
	521--565.
	
\end{thebibliography}
\end{document}